\documentclass[11pt]{amsart}
\usepackage{amsmath}
\usepackage{amssymb}
\usepackage{amscd}
\usepackage{color}
\usepackage{mathtools}
\usepackage{enumitem}

\def\NZQ{\mathbb}               
\def\NN{{\NZQ N}}
\def\QQ{{\NZQ Q}}
\def\ZZ{{\NZQ Z}}
\def\RR{{\NZQ R}}

\newtheorem{Theorem}{Theorem}[section]
\newtheorem{Lemma}[Theorem]{Lemma}
\newtheorem{Corollary}[Theorem]{Corollary}
\newtheorem{Proposition}[Theorem]{Proposition}
\newtheorem{Remark}[Theorem]{Remark}

\newtheorem{Example}[Theorem]{Example}

\newtheorem{Definition}[Theorem]{Definition}

\newtheorem{remark}{Remark}[Theorem] 
\let\epsilon\varepsilon
\let\phi=\varphi
\let\kappa=\varkappa

\def \a {\alpha}

\def \g {\gamma}

\def\MC{\mathcal}

\def\SS{\MC S}
\def\I{\MC I}
\def\J{\MC J}
\def\K{\MC K}
\def\L{\MC L}

\def\v[#1]{v_{_{#1}}} 
\def\asf[#1]{\overline{v}_{\scaleto{#1}{4.5pt}}} 

\def \|{\hspace{0.05in} | \hspace{0.05in}}

\begin{document}
\title{The Asymptotic Samuel Function of a Filtration}

\author{Steven Dale Cutkosky}
\author{Smita Praharaj}

\thanks{The first author was partially supported by NSF grant DMS-2054394.}

\address{Steven Dale Cutkosky, Department of Mathematics,
University of Missouri, Columbia, MO 65211, USA}
\email{cutkoskys@missouri.edu}

\address{Smita Praharaj,  Department of Mathematics,
University of Missouri, Columbia, MO 65211, USA}
\email{smitapraharaj@mail.missouri.edu}

\begin{abstract} We extend the asymptotic Samuel function of an ideal to a filtration and show that many of the good properties of this function for an ideal are true for  filtrations. There are, however, interesting differences, which we explore. We study the notion of projective equivalence of filtrations and the relation between the asymptotic Samuel function and the multiplicity of a filtration. We further consider the case of discrete valued filtrations and show that they have particularly nice properties.
\end{abstract}

\maketitle

\section{Introduction} 
In this paper, we extend the asymptotic Samuel function of an ideal to any arbitrary filtration of a Noetherian ring $R$. The asymptotic Samuel function of an ideal was first defined by Samuel in \cite{Sa}. Its basic properties and some beautiful theorems about it are proven in the articles \cite{Sa}, \cite{N}, \cite{R2}, \cite{R3}, \cite{LT}, \cite{McA} and \cite{MRS}
and are surveyed in the book \cite{HS}.  A recent paper studying the asymptotic Samuel function in the context of resolution of singularities is \cite{BBE}.

Let $R$ be a Noetherian ring.
For a filtration $\I = \{I_m\}_{m \in \NN}$ of ideals in $R$, define the order of $\I$ by  $\nu_{\I}(f) \coloneqq \sup \hspace{0.02in} \{m \mid f\in I_m\}$.
 We define the asymptotic Samuel function of $\mathcal I$  as 
the function $\overline{\nu}_{\I} : R \to \RR_{\geq 0} \cup \{\infty\}$ given by 
$$
\overline{\nu}_{\I}(x) = \lim\limits_{n \to \infty} \dfrac{\nu_{\I}(x^n)}{n}
$$
 for $x \in R$. The existence of this limit is shown in Theorem \ref{LimExist}. 
 If $\mathcal I=\{I^m\}_{m \in \NN}$ is the adic-filtration of powers of an ideal $I$ then the asymptotic Samuel function $\overline{\nu}_{\mathcal I}$ of the filtration $\mathcal I$ is equal to  the classical asymptotic Samuel function $\overline{\nu}_I$ of the ideal $I$.

The Rees algebra of a filtration $\I = \{I_m\}_{m \in \NN}$ is the graded $R$-algebra 
$$
R[\I] = \sum\limits_{m \in \NN} I_m t^m \subseteq R[t],
$$
 where $R[t]$ is the polynomial ring in the variable $t$ over $R$, which is viewed as a graded $R$-algebra where $t$ has degree 1. Let $\overline{R[\I]} = \overline{\sum\limits_{m \in \NN} I_m t^m}$ be the integral closure of $R[\I]$ in $R[t]$. 

 If $\I=\{I^m\}_{m \in \NN}$ is the adic-filtration of powers of an ideal $I$, then  $R[\I]=\bigoplus\limits_{m \in \NN}I^m t^m $ is the usual Rees algebra of $I$, and 
 $\overline{R[\I]}=\bigoplus\limits_{m \in \NN}\overline{I^m} t^m = \overline{R[I]}$, where $\overline{I^m}$ is the integral closure of the ideal $I^m$. 
 
For a general filtration $\I=\{I_m\}_{m \in \NN}$ of a Noetherian ring $R$, the integral closure of the Rees algebra $R[\I]$ is larger than the ring
$\bigoplus\limits_{m \in \NN}\overline{I_m}t^m$. In fact (Lemma \ref{lem19}), the integral closure of $R[\I]$ is 
$$
\overline{R[\I]} = \sum\limits_{m \in \NN} J_m t^m
$$
 where $J_m = \{f \in R \hspace{0.05in} | \hspace{0.05in} f^r \in \overline{I_{rm}} \text{ for some } r > 0\}$ and 
 $\mathcal{IC}(\mathcal I) \coloneqq \{J_m\}_{m \in \NN}$ is a filtration of $R$. 
 

 Given a filtration $\I = \{I_m\}_{m \in \NN}$ of $R$ and $\a \in \RR_{> 0}$, define the twist of $\mathcal I$ by $\alpha$ to be the  filtration 
 $$
 \I^{(\a)} = \{I^{(\a)}_{m}\}_{m \in \NN} = \{I_{\lceil \a m \rceil}\}_{m \in \NN}.
 $$ 
 
 In Theorem \ref{thm23} it is shown that if $\I$ is a filtration and $\a \in \RR_{>0}$, then, 
 $$
 \overline{\nu}_{\I} = \a \hspace{0.02in} \overline{\nu}_{\I^{(\a)}}.
 $$
This is in contrast to the case of an ideal $I$ in $R$, where the set $T$ of positive numbers $\alpha$ for which there exists an ideal $J$ of $R$  such that $\overline{\nu}_{\I}=\alpha \hspace{0.02in}\overline{\nu}_{\J}$ is a discrete subset of $\RR$ (\cite{MRS}, Exercise 10.27 \cite{HS}).

Ideals $I$ and $J$ of a Noetherian ring $R$ are said to be projectively equivalent if there exists $\alpha\in\RR_{> 0}$ such that $\overline{\nu}_{I} = \alpha \hspace{0.04in} \overline{\nu}_{J}$. Corollary  11.9 (ii) \cite{McA}  
or Exercise  10.26 of \cite{HS}  provides a characterization of projectively equivalent ideals in terms of integral closures, which we state in Proposition \ref{prop23} below.

\begin{Proposition}\label{prop23}
Let $I$ and $J$ be ideals in a Noetherian ring $R$ and $\alpha\in \RR_{>0}$. Then the following are equivalent
\begin{enumerate}
\item[1)] $I$ and $J$ are projectively equivalent with $\overline\nu_{I}=\alpha \hspace{0.02in}\overline{\nu_J}$.
\item[2)]  There exists $ m, n \in \ZZ_{>0}$ such that $\alpha=\frac{m}{n}$ and $\overline{I^m} = \overline{J^n}$. 
\item[3)] There exists   $ m, n \in \ZZ_{>0}$ such that $\alpha=\frac{m}{n}$ and
we have equality of  integral closures of  Rees algebras 
$$
\overline{R[I^m]}=\overline{R[J^n]}.
$$
\end{enumerate}
\end{Proposition}

It follows that if $I$ and $J$ are projectively equivalent with  $\overline{\nu}_{I} = \alpha \hspace{0.02in} \overline{\nu}_J$, then $\alpha\in \QQ$.

The definition of projective equivalence for ideals extends naturally to filtrations. 
Filtrations $\I$ and $\J$ in a Noetherian ring $R$ are said to be projectively equivalent if there exists $\a \in \RR_{> 0}$ such that $\overline{\nu}_{\I} = \a \hspace{0.02in} \overline{\nu}_{\J}$.

Suppose that $I$ and $J$ are ideals in a Noetherian ring $R$ and $\I = \{I^n\}_{n \in \NN}$, $\J = \{J^n\}_{n \in \NN}$ are their associated adic-filtrations. We have that $\overline{\nu}_I = \overline{\nu}_{\I}$ and $\overline{\nu}_J = \overline{\nu}_{\J}$, so the ideals $I$ and $J$ are projectively equivalent if and only if the associated adic-filtrations $\I$ and $\J$ are projectively equivalent.

Theorem \ref{thm23} shows that given any $\a \in \RR_{>0}$, and a filtration $\I = \{I_m\}_{m \in \NN}$ of $R$, the twist of $\I$ by $\a$ is projectively equivalent to $\I$ since $\overline{\nu}_{\I} = \a \hspace{0.02in} \overline{\nu}_{\I^{(\a)}}$. Thus, the conclusion of the rationality of $\a$, as shown in Proposition \ref{prop23} for projective equivalence of ideals, does not extend to filtrations.

We provide the following necessary and sufficient condition for projective equivalence of filtrations.
\begin{Theorem}(Theorem \ref{thm25})\label{IntroThm}
Let $\I = \{I_m\}_{m \in \NN}$ and $\J = \{J_m\}_{m \in \NN}$ be filtrations in a Noetherian ring $R$. Then $\I$ and $\J$ are projectively equivalent if and only if $\exists$ $\alpha$, $\beta$ $\in \RR_{>0}$ such that $\mathcal{IC}(\I^{(\alpha)}) = \mathcal{IC}(\J^{(\beta)})$, or equivalently, $\overline{R[\I^{(\a)}]} = \overline{R[\J^{(\beta)}]}$.
\end{Theorem}

 We give an example in  Example \ref{ex26} of filtrations   $\I$ and $\J$ which are projectively equivalent  with $\overline{\nu}_{\I} = \overline{\nu}_{\J}$ but for no $\alpha$ or $\beta$ $\in \QQ$ do we have that $\overline{R[\I^{(\a)}]} = \overline{R[\J^{(\beta)}]}$. 
 
 In the case that  $\mathcal I$ and $\mathcal J$ are adic-filtrations of powers of ideals, we have by Proposition \ref{prop23} that 
 $\mathcal I$ and $\mathcal J$ are projectively equivalent if and only if $\overline{R[\I^{(m)}]} = \overline{R[\J^{(n)}]}$ for $m,n\in \ZZ_{>0}$ with $\overline{\nu}_{\I} = \frac{m}{n} \hspace{0.05in} \overline{\nu}_{\J}$. In this case,
 $\overline{R[\mathcal I^{(m)}]}=\overline{R[I^m]}$ and $\overline{R[\mathcal \J^{(n)}]}=\overline{R[J^n]}$.

 If $I$ and $J$ are ideals such that $\overline \nu_I=\overline \nu_J$, then $\overline{R[I]}=\overline{R[J]}$ by Corollary 6.9.1 \cite{HS} (Stated in Lemma \ref{LemId} of this paper). Thus $\mbox{Proj}(\overline{R[I]})\cong \mbox{Proj}(\overline{R[J]})$ as $R$-schemes. However, this property fails for general filtrations as we now show. 
 
 Let $k$ be a field and $R=k[[x]]$, with maximal ideal $m_R=(x)$. Let $\mathcal I=\{I_m\}_{m \in \NN}$ and $\mathcal J=\{J_m\}_{m \in \NN}$ where $I_m=(x^{m+1})$ and $J_m=(x^m)$ $ \forall $ $m > 0$ and $I_0 = J_0 = R$. In Example \ref{ex26}, it is shown that $\overline{\nu}_{\mathcal I}=\overline{\nu}_{\mathcal J}$. We have that both $R[\mathcal I]$ and $R[\mathcal J]$ are integrally closed. However, $\mbox{Proj}(R[\mathcal I])$ is not isomorphic to $\mbox{Proj}(R[\mathcal J])$. To show this we use the theory of analytic spread of filtrations \cite{CS}. The analytic spread of a filtration $\mathcal F=\{F_m\}_{m \in \NN}$ of a Noetherian local ring $R$ is defined (in equation (6) of \cite{CS})    to be
 $$
 \ell(R)=\dim R[\mathcal F]/m_RR[\mathcal F].
 $$
 The analytic spread has the geometric interpretation that  $\ell(\mathcal F)=\dim \pi_{\mathcal F}^{-1}(m_R)+1$ where $\pi_{\mathcal F}:\mbox{Proj}(R[\mathcal F])\rightarrow \mbox{Spec}(R)$ is the natural projection. 
  
 By Lemma 3.8 \cite{CS}, we have that $\ell(\mathcal F)=0$ if and only if for all $n>0$ and $f\in F_n$
 there exists $m>0$ such that $f^m\in m_RF_{mn}$. We verify that this condition holds for $\mathcal I$. Suppose that $f\in I_n$. Then ${\rm ord}(f)\ge n+1$. For $m\ge 2$, ${\rm ord}(f^m)\ge nm+m\ge (nm+1)+1$ and so $f^m\in m_RI_{nm}$. Thus $\ell(\mathcal I)=0$ and so $\pi_{\mathcal I}^{-1}(m_R)=\emptyset$. Since $\mathcal I$ is a filtration of $m_R$-primary ideals, we then have that $\mbox{Proj}(R[\mathcal I])\cong\mbox{Spec}(R)\setminus\{m_R\}$. In contrast, $R[\mathcal J]$ is the Rees algebra of a principal ideal, so  $\mbox{Proj}(R[\mathcal J])\cong \mbox{Spec}(R)$.
 
From Prop \ref{prop23} we obtain  the geometric interpretation of the condition that ideals $I$ and $J$ are projectively equivalent; we have that   $\mbox{Proj}(\overline{R[I^{m}]})=\mbox{Proj}(\overline{R[J^{n}]})$. The algebras  $R[I^{m}]$ and  $R[J^{n}]$ are  suitable Veronese algebras of $R[I]$ and $R[J]$, which are the Rees algebras of twists of the corresponding adic-filtrations by the integers $m$ and $n$.  From Theorem \ref{IntroThm}, we obtain the statement that if $\mathcal I$ and $\mathcal J$ are filtrations which are projectively equivalent, then    by taking  suitable twists by real numbers $\alpha$ and $\beta$, we have that $\mbox{Proj}(\overline{R[\mathcal I^{(\alpha)}]})=\mbox{Proj}(\overline{R[\mathcal J^{(\beta)}]})$. 
 
 We prove the following theorem, which shows that given a filtration $\mathcal I$, there is a unique largest filtration $\mathcal K(\mathcal I)$ such that $\mathcal I$ and $\mathcal K(\mathcal I)$ have the same asymptotic Samuel function.
 
\begin{Theorem}(Theorem \ref{thm31})
For a filtration $\I = \{I_m\}_{m \in \NN}$ of ideals in $R$, define
$$
K(\I)_m \coloneqq \{f \in R \hspace{0.05in} | \hspace{0.05in} \overline{\nu}_{\I}(f) \geq m\} \hspace{0.05in} \forall \hspace{0.05in} m \in \NN.$$
Then $\K(\I) \coloneqq \{K(\I)_m\}_{m \in \NN}$ is a filtration of ideals in $R$ and $\I \subseteq \K(\I)$. Moreover, $\overline{\nu}_{\I} = \overline{\nu}_{\K(\I)}$ and $\K(\I)$ is the unique, largest filtration $\J$ such that $\overline{\nu}_{\I}= \overline{\nu}_{\J}$.
\end{Theorem} 
 
 If $\mathcal I=\{I^m\}_{m \in \NN}$ is the adic-filtration of powers of an ideal, then $\mathcal K(\mathcal I)=\{\overline{I^m}\}_{m \in \NN}$, the filtration  of integral closures of powers of $I$ (by Lemma \ref{LemId}). 
 
 In contrast, for a general filtration, it is possible for $\mathcal K(\mathcal I)$ to be larger than the filtration $\mathcal{IC}(\mathcal{I})$, the integral closure of $\mathcal I$. Such an example is given in Example \ref{Newex}. By Lemma \ref{KIntcl}, the Rees algebra $R[\mathcal K(\mathcal I)]$ is integrally closed. Thus 
 for a filtration $\mathcal I$, we have inclusions of Rees algebras
 \begin{equation}\label{eqInt1}
 R[\mathcal I]\subseteq R[\mathcal{IC}(\mathcal I)]=\overline{R[\mathcal I]}\subseteq R[\mathcal K(\mathcal I)]=\overline{R[\mathcal K(\mathcal I)]}
 \end{equation}
 where the two inclusions can be proper. In Theorem \ref{NewThm}  we give another characterization of projective equivalence.
 
 \begin{Theorem}(Theorem \ref{NewThm})
Suppose $\I$ and $\J$ are filtrations of a Noetherian ring $R$. Then $\I$ is projectively equivalent to $\J$ with $\overline{\nu}_{\I} = \a \hspace{0.02in} \overline{\nu}_{\J}$ if and only if $\K(\I^{(\a)}) = \K(\J)$.
\end{Theorem} 

In Section \ref{DVF}, we consider discrete valued filtrations (defined at the beginning of Section \ref{DVF}). We generalize some of the theory of Rees valuations of ideals (\cite{R2}, \cite{R3}, Section 10 \cite{HS}) to these filtrations. 

If $\mathcal I=\{I^m\}_{m \in \NN}$ is the adic-filtration of powers of an ideal $I$, and $v_1,\ldots,v_s$ are the Rees valuations of $I$, then for $f\in R$,
\begin{equation}\label{eqInt2}
\overline\nu_I(f)=\min\left\{\frac{v_1(f)}{v_1(I)},\ldots,\frac{v_s(f)}{v_s(I)}\right\}.
\end{equation}
This result is proven in \cite{R2} and after Lemma 10.1.5 in \cite{HS}.

We prove the following Lemma, which generalizes this result to discrete valued filtrations.  The Rees algebras of discrete valued filtrations are generally non Noetherian $R$-algebras. Formula (\ref{eqInt2}) generalizes to these filtrations.

\begin{Lemma}(Lemma \ref{lem30})
Let $\I=\{I_m\}$ where $I_m = I(v_1)_{ma_1} \cap \cdots \cap I(v_s)_{ma_s}$
 be a discrete valued filtration of a Noetherian ring $R$. For $f \in R \setminus \{0\}$,
$$\nu_{\I}(f) = \min \bigg\{ \bigg\lfloor \dfrac{v_{1}(f)}{a_1}\bigg\rfloor, \cdots, \bigg\lfloor \dfrac{v_{s}(f)}{a_s}\bigg\rfloor  \bigg\} \hspace{0.1in} \text{ and } \hspace{0.1in} \overline{\nu}_{\I}(f) = \min \bigg\{  \dfrac{v_{1}(f)}{a_1}, \cdots, \dfrac{v_{s}(f)}{a_s}\bigg\}.$$
\end{Lemma}

In Theorem \ref{Theorem1}, we generalize to discrete valued  filtrations the proof of uniqueness of Rees valuations for ideals given in Theorem 10.1.6 \cite{HS}. We obtain the following Corollary.

\begin{Corollary}(Corollary \ref{CorIrr})
Let $\I = \{I_m\}_{m \in \NN}$ and $\J = \{J_m\}_{m \in \NN}$ be discrete valued filtrations of a Noetherian ring $R$, where $I_m = \bigcap\limits_{i=1}^s I(v_i)_{a_i m}$ and $J_m = \bigcap\limits_{i=1}^r I(v'_{i})_{a'_i m}$ $ \forall $ $m \in \NN$ are irredundant representations. If $\overline{\nu}_{\I} = \overline{\nu}_{\J}$, then $r = s$ and after reindexing, $a_i = a'_i$ and $v_i = v'_i$.
\end{Corollary}

If $\I = \{I_m\}_{m \in \NN}$ where $I_m=\bigcap\limits_{i=1}^sI(v_i)_{a_im}$, then 
$\mathcal I^{[\alpha]}$ is the filtration $I^{[\alpha]}_m=\bigcap\limits_{i=1}^sI(v_i)_{\alpha m a_i}$. This filtration is well defined (independent of representation of $\mathcal I$ as a discrete valued filtration).

\begin{Proposition}(Proposition \ref{cor37}) 
Suppose that $\mathcal I$ is a discrete valued filtration of a Noetherian ring $R$ and $\alpha\in \RR_{>0}$. Then $\K(\I^{(\a)}) = \I^{[\a]} = \K(\I^{[\a]})$. 
\end{Proposition}

In particular, the chain of inclusions of (\ref{eqInt1}) are all equalities for discrete valued filtrations.

\begin{Theorem}(Theorem \ref{Thmsquare}) 
Let $\I=\{I_m\}_{m \in \NN}$ and $\J=\{J_m\}_{m \in \NN}$ be discrete valued filtrations of a Noetherian ring $R$ and $\a\in \RR_{>0}$. Then $\overline{\nu}_{\I} = \a \hspace{0.05in} \overline{\nu}_{\J}$ if and only if $\J=\I^{[\a]}$.
\end{Theorem}

In the final section, we compare the asymptotic Samuel function of a filtration with its multiplicity. 
The multiplicity of an $m_R$-primary filtration is defined, and exists as a limit  in an analytically unramified (Noetherian) local ring, but does not exist in general if the ring is not generically analytically unramified. This follows from Theorem 1.1 \cite{SDC2}.

In an analytically irreducible local ring $(R,m_R)$, for any filtration $\mathcal I$ of $m_R$-primary ideals, there is a unique largest filtration $\tilde{\mathcal I}$ containing $\mathcal I$ such that $\mathcal I$ and $\tilde{\mathcal I}$ have the same multiplicity. This is shown in Theorem 1.4 \cite{BLQ}.

For an ideal $I$, we have that the integral closure  $\overline I$ is the largest ideal which has the same asymptotic Samuel function as $I$ (Lemma \ref{LemId}) and this is the largest ideal containing $I$ which has the same multiplicity as $I$ by a theorem of Rees (\cite{R1} and Theorem 11.3.1 \cite{HS}).



For general $m_R$-filtrations $\mathcal I$, we have that $\mathcal I$ and $\mathcal K(\mathcal I)$ have the same multiplicity (Theorem \ref{Mult}) but we give an example in Example \ref{MultEx}  showing that $\tilde{\mathcal I}$ can be much larger than $\mathcal K(\mathcal I)$. In the case that $\mathcal I$ is a discrete valued filtrations we have that $\mathcal I = \mathcal{K(I)} = \tilde{\mathcal I}$.

\section{Notation}

We will denote the nonnegative integers by $\NN$ and the positive integers by $\ZZ_{>0}$, the set of rational numbers by $\QQ$, the set of nonnegative rational numbers  by $\QQ_{\ge 0}$  and the positive rational numbers by $\QQ_{>0}$. 
We will denote the set of nonnegative real numbers by $\RR_{\ge0}$ and the positive real numbers by $\RR_{>0}$. 

The round down $\lfloor x \rfloor$ of a real number $x$ is the largest integer which is less that or equal to $x$. The round up $\lceil x\rceil$ of a real number $x$ is the smallest integer which is greater than or equal to $x$.

The maximal ideal of a local ring $R$ will be denoted by $m_R$. The integral closure of an ideal $I$ in a Noetherian ring $R$ will be denoted by $\overline I$. The multiplicative group of units in a ring $R$ will be denoted by $R^{\times}$.

\section{The asymptotic Samuel function of a Filtration}

Let $R$ be a Noetherian ring.
We extend the asymptotic Samuel function of an ideal to an arbitrary filtration of $R$.
Let $\I = \{I_m\}_{m \in \NN}$ be a filtration of ideals in $R$, that is, $I_0 = R$, $I_n$ is an ideal in $R$, $I_n \supseteq I_{n+1}$ and $I_n \cdot I_m \subseteq I_{n+m}$, $\forall$ $m, n \in \NN$.

We say that a filtration $\mathcal I=\{I_m\}_{m \in \NN}$ is a subset of a filtration $\mathcal J=\{J_m\}_{m \in \NN}$ and write $\mathcal I\subseteq \mathcal J$ if $I_m\subseteq J_m$ $ \forall $ $ m \in \NN$.

The following filtration will be important in this paper. 

\begin{Definition}\label{uptwist}
Given a filtration $\I = \{I_m\}_{m \in \NN}$ of $R$  and $\a \in \RR_{> 0}$, define a filtration $\I^{(\a)} = \{I^{(\a)}_{m}\}_{m \in \NN}$ by  $I^{(\alpha)}_m = I_{\lceil \a m \rceil}$. We will call $\I^{(\a)}$ {\textbf {the twist of $\mathcal I$ by $\alpha$}}.

 \end{Definition}


\begin{Definition}
For a filtration $\I = \{I_m\}_{m \in \NN}$ of $R$, define a function $\nu_{\I} : R \to \NN \cup \{\infty\}$ by $\nu_{\I}(f) \coloneqq \sup \hspace{0.05in} \{m \hspace{0.05in} | \hspace{0.05in} f \in I_m\}$. We call this the \textbf{order} of $\I$.
\end{Definition}

\begin{Remark}\label{rmk3.3}
For $x, y \in R$, $\nu_{\I}(xy) \geq \nu_{\I}(x) + \nu_{\I}(y)$ and $\nu_{\I}(x+y) \geq \min\{ \nu_{\I}(x), \nu_{\I}(y)\}$. Observe that $\nu_{\I}(f) = \infty$ if and only if $f \in \bigcap\limits_{m \in \NN} I_m$.
\end{Remark} 

\begin{Theorem}\label{LimExist}
Let $\I = \{I_m\}_{m \in \NN}$ be a filtration of ideals in a Noetherian ring $R$. For any $x \in R$, the limit $\lim\limits_{n \to \infty} \dfrac{\nu_{\I}(x^n)}{n}$ exists as an element of $\RR_{\ge 0}\cup \{\infty\}$.
\end{Theorem}

\begin{proof}
Let $x \in R$ and $u \coloneqq \limsup\limits_{n \to \infty} \dfrac{\nu_{\I}(x^n)}{n}$ (which could possibly be $\infty$). 

If $u = 0$, then the limit exists since
$0 \leq \liminf\limits_{n \to \infty} \dfrac{\nu_{\I}(x^n)}{n} \leq \limsup\limits_{n \to \infty} \dfrac{\nu_{\I}(x^n)}{n} = 0$.

Assume $u > 0$. Let $N \in \RR_{> 0}$ be such that $N < u$. We can choose $n_0 \in \ZZ_{>0}$ such that $\dfrac{\nu_{\I}(x^{n_0})}{n_0} > N$. Let $n$ be any arbitrary positive integer. We have $n = q n_0 + r$ for some $q, r \in \NN$ such that $0 \leq r < n_0$. 

Using Remark \ref{rmk3.3} it follows  that
$$\dfrac{\nu_{\I}(x^n)}{n} = \dfrac{\nu_{\I}(x^{qn_0+r})}{qn_0+r} \geq \dfrac{\nu_{\I}(x^{qn_0})}{qn_0+r} + \dfrac{\nu_{\I}(x^r)}{qn_0+r} \geq q \dfrac{\nu_{\I}(x^{n_0})}{qn_0+r} + \dfrac{\nu_{\I}(x^r)}{qn_0+r} \geq q \dfrac{\nu_{\I}(x^{n_0})}{qn_0+r}.$$

$$\text{ This implies \hspace{0.02in}} \dfrac{\nu_{\I}(x^n)}{n} \geq \dfrac{qn_0}{qn_0+r} \dfrac{\nu_{\I}(x^{n_0})}{n_0} \geq \dfrac{qn_0}{qn_0+r} N \geq \dfrac{qn_0}{qn_0+n_0} N = \dfrac{q}{q+1} N.$$

Taking $\liminf$ on both sides, we get $\liminf\limits_{n \to \infty} \dfrac{\nu_{\I}(x^n)}{n} \geq \liminf\limits_{n \to \infty} \dfrac{q}{q+1} N$.

Clearly, $\liminf\limits_{n \to \infty} \dfrac{q}{q+1} \leq 1$. Since $r < n_0$, $\dfrac{q}{q+1} = \dfrac{n-r}{n+n_0 -r} \geq \dfrac{n - n_0}{n}$ implying $\liminf\limits_{n \to \infty} \dfrac{q}{q+1 } \geq 1$. Thus $\liminf\limits_{n \to \infty} \dfrac{q}{q+1} = 1$. This shows that $\liminf\limits_{n \to \infty} \dfrac{\nu_{\I}(x^n)}{n} \geq N$ for any positive real number strictly smaller than $u$. Since $N$ was arbitrarily chosen, $\liminf\limits_{n \to \infty} \dfrac{\nu_{\I}(x^n)}{n} \geq \limsup\limits_{n \to \infty} \dfrac{\nu_{\I}(x^n)}{n}$, proving that the limit exists.
\end{proof}

\begin{Definition}
For a filtration $\I = \{I_m\}_{m \in \NN}$ of ideals in $R$, we define the function $\overline{\nu}_{\I} : R \to \RR_{\geq 0} \cup \{\infty\}$ by $\overline{\nu}_{\I}(x) \coloneqq \lim\limits_{n \to \infty} \dfrac{\nu_{\I}(x^n)}{n}$ for  $x \in R$.
\end{Definition}

The asymptotic Samuel function of an ideal $I$ in a Noetherian ring $R$ is defined  to be $\overline \nu_{I}(x)=\lim\limits_{n\to\infty}
\dfrac{\mbox{ord}_I(x^n)}{n}$ where $\mbox{ord}_I(x)=\sup \hspace{0.02in} \{m\mid x\in I^m\}$ for $x\in R$.
Then for any $x \in R$, $\overline{\nu}_I(x) = \overline{\nu}_{\I}(x)$, where $\mathcal I$ is the adic-filtration $\I = \{I^m\}_{m \in \NN}$.
This follows since $ord_I(x) = \nu_{\I}(x)$ for any $x \in R$.

Thus, $\overline{\nu}_{\I}$ extends the concept of the asymptotic Samuel function of an ideal to an arbitrary filtration of a Noetherian ring. We call $\overline{\nu}_{\I}$ the \textbf{asymptotic Samuel function of the filtration $\I$}.

An important property of the asymptotic Samuel function of an ideal is the following Lemma, which is proven in Corollary 6.9.1 \cite{HS}.

\begin{Lemma}\label{LemId} Let $R$ be a Noetherian ring, $I$ an ideal in $R$ and $c\in \NN$. Then
$$
\{x\in R\mid \overline{\nu}_{I}(x)\ge c\} =\overline{I^c}.
$$
\end{Lemma}

\begin{Remark}\label{rmk2.7}
Let $\I \subseteq \J$ be filtrations. Then $\overline{\nu}_{\I} \leq \overline{\nu}_{\J}$.
\end{Remark}

\begin{proof}
 For $x \in R$, we have that $\nu_{\I}(x^i) \leq \nu_{\J}(x^i)$ $ \forall $ $ i \in \NN$ so that 
 $\overline{\nu}_{\I}(x) \leq \overline{\nu}_{\J}(x)$.
\end{proof}

\begin{Example}
In a Noetherian local ring $(R, m_R)$, consider the filtration $\I = \{I_m\}_{m \in \NN}$ given by $I_0 = R$ and $I_m = m_R$ $ \forall$ $ m > 0$. In this case, $\overline{\nu}_{\I}(a) = \infty$ if $a \in m_R$ and $\overline{\nu}_{\I}(a) = 0$ if $a \notin m_R$.
\end{Example}

\begin{Remark}
The range of $\overline{\nu}_{\I}$ may not be contained in $\QQ_{\geq 0} \cup \{\infty\}$. This follows from Theorem \ref{thm23} below. Observe that this is different from the case when $\I = \{I^m\}_{m \in \NN}$ for an ideal $I$ of $R$, in which case we do have $\overline{\nu}_{\I} = \overline{\nu}_{I}$ and then the range of $\overline{\nu}_{\I}$ is contained in  $\QQ_{\geq 0} \cup \{\infty\}$ (as shown after Lemma 10.1.5 \cite{HS}).
\end{Remark}

\begin{Theorem}\label{thm23}
Let $\I = \{I_m\}_{m \in \NN}$ be a filtration in $R$ and $\a \in \RR_{>0}$. Then, $\overline{\nu}_{\I} = \a \hspace{0.02in} \overline{\nu}_{\I^{(\a)}}$, where $\I^{(\a)} = \{I^{(\a)}_{m}\}_{m \in \NN} = \{I_{\lceil \a m \rceil}\}_{m \in \NN}$.
\end{Theorem}

\begin{proof}
For $x \in R$ and $i \in \NN$, $x^i \in I_{\lceil \a \hspace{0.01in} \nu_{\I^{(\a)}}(x^i) \rceil}$, so, $\nu_{\I}(x^i) \geq \lceil \a \hspace{0.01in} \nu_{\I^{(\a)}}(x^i) \rceil$ which gives $\overline{\nu}_{\mathcal I}(x)=\lim\limits_{i \to \infty} \dfrac{\nu_{\I}(x^i)}{i} \geq \lim\limits_{i \to \infty} \dfrac{\lceil \a \hspace{0.01in} \nu_{\I^{(\a)}}(x^i) \rceil}{i}$. Since  $\a \hspace{0.01in} \nu_{\I^{(\a)}}(x^i) \leq \lceil \a \hspace{0.01in} \nu_{\I^{(\a)}}(x^i) \rceil \leq \a \hspace{0.01in} 
\nu_{\I^{(\a)}}(x^i) + 1$,
$$\lim\limits_{i \to \infty} \dfrac{\a \hspace{0.01in} \nu_{\I^{(\a)}}(x^i)}{i} \leq \lim\limits_{i \to \infty} \dfrac{\lceil \a \hspace{0.01in} \nu_{\I^{(\a)}}(x^i) \rceil}{i} \leq \lim\limits_{i \to \infty} \dfrac{\a \hspace{0.01in} \nu_{\I^{(\a)}}(x^i) + 1}{i}$$

This implies $\lim\limits_{i \to \infty} \dfrac{\lceil \a \hspace{0.01in} \nu_{\I^{(\a)}}(x^i) \rceil}{i} = \a \hspace{0.02in} \overline{\nu}_{\I^{(\a)}}(x)$. Thus, $\overline{\nu}_{\I}(x) \geq \a \hspace{0.02in} \overline{\nu}_{\I^{(\a)}}(x)$. 

Note that if $x \in I_k$ for some $k \in \NN$, then $x \in I_{\big\lceil \big\lfloor \frac{k}{\a} \big\rfloor \a \big\rceil}$. It follows that $\nu_{\I^{(\a)}}(x) \geq \Big\lfloor \dfrac{\nu_{\I}(x)}{\a} \Big\rfloor$. Thus,  $ \forall $ $ i \in \ZZ_{>0}$,
$\dfrac{\nu_{\I^{(\a)}}(x^i)}{i} \geq \dfrac{\Big\lfloor \dfrac{\nu_{\I}(x^i)}{\a} \Big\rfloor}{i}$ which implies $\lim\limits_{i \to \infty} \dfrac{\nu_{\I^{(\a)}}(x^i)}{i} \geq \lim\limits_{i \to \infty} \dfrac{\Big\lfloor \dfrac{\nu_{\I}(x^i)}{\a} \Big\rfloor}{i}$. Since $\Big\lfloor \dfrac{\nu_{\I}(x^i)}{\a} \Big\rfloor > \dfrac{\nu_{\I}(x^i)}{\a} - 1$, $\lim\limits_{i \to \infty} \dfrac{\Big\lfloor \dfrac{\nu_{\I}(x^i)}{\a} \Big\rfloor}{i} \geq \lim\limits_{i \to \infty} \dfrac{\dfrac{\nu_{\I}(x^i)}{\a} - 1}{i} = \dfrac{\overline{\nu}_{\I}(x)}{\a}$. 

This shows $\overline{\nu}_{\I^{(\a)}}(x) \geq \dfrac{\overline{\nu}_{\I}(x)}{\a}$, thus proving the result.
\end{proof}

\begin{Proposition}\label{prop14}
Let $\mathcal I$ be a filtration of $R$. For $f, g \in R$,
\begin{enumerate}
    \item $\overline{\nu}_{\I}(f^n) = n \hspace{0.05in} \overline{\nu}_{\I}(f)$ $ \forall $ $n \in \ZZ_{>0}$.
    \item $\overline{\nu}_{\I}(f + g) \geq \min\{\overline{\nu}_{\I}(f), \overline{\nu}_{\I}(g)\}$.
\end{enumerate}
\end{Proposition}

\begin{proof}
Since the limit defining $\overline{\nu}_{\mathcal I}$ exists, any subsequence also converges to the same limit. Thus,
$$ \overline{\nu}_{\mathcal I}(f) = \lim\limits_{m \to \infty} \dfrac{\nu_{\I}(f^m)}{m} = \lim\limits_{m \to \infty} \dfrac{\nu_{\I}(f^{nm})}{nm} = \dfrac{1}{n} \lim\limits_{m \to \infty} \dfrac{\nu_{\I}((f^n)^m)}{m} = \dfrac{1}{n} \hspace{0.05in} \overline{\nu}_{\I}(f^n), \hspace{0.05in} \forall \hspace{0.05in} n \in \ZZ_{>0}.$$
This proves $(1)$.

 To prove $(2)$, let $f, g \in R$ be such that $\overline{\nu}_{\I}(f) \geq \overline{\nu}_{\I}(g)$.
 
 For $\epsilon > 0$, $ \exists$ $m_0 \in \ZZ_{>0}$ such that $\forall$ $m \geq m_0$, $\dfrac{\nu_{\I}(f^m)}{m} \geq \overline{\nu}_{\I}(g) - \epsilon$ and $\dfrac{ \nu_{\I}(g^m)}{m} \geq \overline{\nu}_{\I}(g) - \epsilon$. For all $m, k \in \ZZ_{>0}$, $\nu_{\I}((f+g)^{mk}) \geq \min\limits_{i+j = mk} \{\nu_{\I}(f^i g^j)\}$, using Remark \ref{rmk3.3}. Since $i + j = mk, mk \geq m \bigg\lfloor  \dfrac{i}{m} \bigg\rfloor + m \bigg\lfloor  \dfrac{j}{m} \bigg\rfloor \geq mk - 2m$. 
 Thus by Remark \ref{rmk3.3},
 $$\nu_{\I}(f^ig^j) \geq \nu_{\I}(f^i) + \nu_{\I}(g^j) \geq \nu_{\I}(f^{m \lfloor  \frac{i}{m} \rfloor }) + \nu_{\I}(g^{m
\lfloor  \frac{j}{m} \rfloor }) \geq \bigg\lfloor  \dfrac{i}{m} \bigg\rfloor \nu_{\I}(f^m) + \bigg\lfloor  \dfrac{j}{m} \bigg\rfloor \nu_{\I}(g^m).$$
$$\text{For } m \geq m_0, \nu_{\I}(f^ig^j)  \geq \bigg\lfloor  \dfrac{i}{m} \bigg\rfloor m(\overline{\nu}_{\I}(g) - \epsilon) + \bigg\lfloor  \dfrac{j}{m} \bigg\rfloor m(\overline{\nu}_{\I}(g) - \epsilon) \geq (mk - 2m)(\overline{\nu}_{\I}(g) - \epsilon).$$

Thus, for $k >> 0$,
$$\dfrac{\nu_{\I}((f+g)^{mk})}{mk} \geq \dfrac{mk(\overline{\nu}_{\I}(g) - \epsilon) - 2m (\overline{\nu}_{\I}(g) - \epsilon)}{mk} = \overline{\nu}_{\I}(g) - \epsilon - \dfrac{2}{k} (\overline{\nu}_{\I}(g) - \epsilon).$$

Taking limits as $k \to \infty$, we get $\overline{\nu}_{\I}(f+g) \geq \overline{\nu}_{\I}(g) - \epsilon.$ Since $\epsilon$ is arbitrary, $\overline{\nu}_{\I}(f + g) \geq \overline{\nu}_{\I}(g) = \min \{\overline{\nu}_{\I}(f), \overline{\nu}_{\I}(g)\}$. This completes the proof.
\end{proof}

\section{Integral closure}

\begin{Lemma}\label{lem21}
For an ideal $I$ in a Noetherian ring $R$, $\overline{\nu}_{\I} = \overline{\nu}_{\overline{\I}}$ where $\I = \{I^m\}_{m \in \NN}$ and $\overline{\I} = \{\overline{I^m}\}_{m \in \NN}$.
\end{Lemma}

\begin{proof}
For $x \in R$ and $i \in \NN$, $x^i \in \overline{I^{\nu_{\overline{\I}}(x^i)}}$ which gives $\overline{\nu}_{\I}(x^i) \geq \nu_{\overline{\I}}(x^i)$, by Lemma \ref{LemId}. By Proposition \ref{prop14}(a), $\overline{\nu}_{\I}(x) \geq \dfrac{\nu_{\overline{\I}}(x^i)}{i}$ $ \forall $ $ i \in \ZZ_{>0}$. This implies $\overline{\nu}_{\I}(x) \geq \lim\limits_{i \to \infty} \dfrac{\nu_{\overline{\I}}(x^i)}{i} = \overline{\nu}_{\overline{\I}}(x)$.

Since $\I \subseteq \overline{\I}$, by Remark \ref{rmk2.7}, $\overline{\nu}_{\I} = \overline{\nu}_{\overline{\I}}$.
\end{proof}

The integral closure of the Rees ring $R[I]$ of an ideal $I$ is $\overline{R[I]}=\sum\limits_{n \in \NN}\overline{I^n}t^n$ (Proposition 5.2.1 \cite{HS}).
We can extend this concept of integral closure to arbitrary filtrations of a Noetherian ring.

\begin{Definition}
The \textbf{Rees algebra of a filtration} $\I = \{I_m\}_{m \in \NN}$ is the graded $R$-algebra $R[\I] = \sum\limits_{m \in \NN} I_m t^m \subseteq R[t]$, where $R[t]$ is the polynomial ring in the variable $t$ over $R$, which is viewed as a graded $R$-algebra where $t$ has degree 1.
Let $\overline{R[\I]} = \overline{\sum\limits_{m \in \NN} I_m t^m}$ be the integral closure of $R[\I]$ in $R[t]$.  
\end{Definition}

In \cite{SDC1} Lemma 5.6, there is a characterization of $\overline{R[\I]}$ when $(R, m_R)$ is a (Noetherian) local ring and $\I$ is an $m_R$-filtration ($I_m$ is an $m_R$-primary ideal for $m>0$). The proof extends to the case where $\I$ is an arbitrary filtration of a Noetherian ring $R$. 

\begin{Lemma}\label{lem19}
Let $\I = \{I_m\}_{m \in \NN}$ be a filtration in $R$. Then $\overline{R[\I]} = \sum\limits_{m \in \NN} J_m t^m$ where $J_m = \{f \in R \hspace{0.05in} | \hspace{0.05in} f^r \in \overline{I_{rm}} \text{ for some } r > 0\}$ and $\mathcal{IC}(\I) \coloneqq \{J_m\}_{m \in \NN}$ is a filtration in $R$.
\end{Lemma}

\begin{Definition}
We call the filtration $\mathcal{IC}(\I)$ defined in Lemma \ref{lem19} the \textbf{integral closure of the filtration $\I$}.
\end{Definition} 

If $\I = \{I^m\}_{m \in \NN}$ is the adic-filtration of an ideal $I$ of $R$, then $\mathcal{IC}(\mathcal I) = \{\overline{I^m}\}_{m \in \NN}$. In this particular case, we have already shown in Lemma \ref{lem21} that $\overline{\nu}_{\I} = \overline{\nu}_{\mathcal{IC}(\mathcal I)}$. In fact, this is true for any arbitrary filtration as well.

\begin{Theorem}\label{finase}
Let $\I = \{I_m\}_{m \in \NN}$ be a filtration of $R$. Then $\overline{\nu}_{\I} = \overline{\nu}_{\mathcal{IC}(\I)}$.
\end{Theorem}

\begin{proof} Let $\mathcal{IC}(\I)=\{J_m\}_{m \in \NN}$ (as in Lemma \ref{lem19}). Since $\I \subseteq \mathcal{IC}(\I)$, $\overline{\nu}_{\I} \leq \overline{\nu}_{\mathcal{IC}(\I)}$, by Remark \ref{rmk2.7}.

Suppose $x \in J_m$ for some $m \in \NN$, that is, $x^r \in \overline{I_{rm}}$ for some $r > 0$. The ideal $I_{rm}$ is a reduction of $\overline{I_{rm}}$ by Corollary 1.2.5 \cite{HS}. By Remark 1.2.3  \cite{HS}, $\exists$ $n \in \ZZ_{>0}$ such that $\forall$ $k \geq n$, $x^{rk} \in (\overline{I_{rm}})^k\subset I^{k-n+1} \subseteq I_{rm(k-n+1)}$. This shows that $\nu_{\I}(x^{rk}) \geq rm(k-n+1)$, which implies $\lim\limits_{k \to \infty} \dfrac{\nu_{\I}(x^{rk})}{rk} \geq \lim\limits_{k \to \infty} \dfrac{rm(k-n+1)}{rk} = m$. Thus, if $x \in J_m$, $\overline{\nu}_{\I}(x) \geq m$. 

For $i \in \NN$, since $x^i \in J_{\nu_{\mathcal{IC}(\I)}(x^i)}, \overline{\nu}_{\I}(x^i) \geq \nu_{\mathcal{IC}(\I)}(x^i)$. By Proposition \ref{prop14}, $\overline{\nu}_{\I}(x) \geq \dfrac{\nu_{\mathcal{IC}(\I)}(x^i)}{i}$ $ \forall$ $i \in \ZZ_{>0}$. Thus, $\overline{\nu}_{\I}(x) \geq \lim\limits_{i \to \infty} \dfrac{\nu_{\mathcal{IC}(\I)}(x^i)}{i} = \overline{\nu}_{\mathcal{IC}(\I)}(x)$. This proves $\overline{\nu}_{\I}(x) = \overline{\nu}_{\mathcal{IC}(\I)}(x)$ $ \forall $ $x \in R$.
\end{proof}

\begin{Corollary}\label{ICcor}
Let $\I = \{I_m\}_{m \in \NN}$ be a filtration in $R$ and $\overline{\I} = \{\overline{I_m}\}_{m \in \NN}$. Then $\overline{\nu}_{\I} = \overline{\nu}_{\overline{\I}}$.
\end{Corollary}

\begin{proof}
This follows from Theorem \ref{finase} and Remark \ref{rmk2.7} since $I_m \subseteq \overline{I_m} \subseteq J_m$ $ \forall $ $m \in \NN$, where $\mathcal{IC}(\mathcal I)=\{J_m\}_{m \in \NN}$ is the integral closure of the filtration $\I$.
\end{proof}

\section{Projective Equivalence}

\begin{Definition}
We define filtrations $\I$ and $\J$ in a Noetherian ring $R$ to be \textbf{projectively equivalent} if there exists $\a \in \RR_{> 0}$ such that $\overline{\nu}_{\I} = \a \hspace{0.02in} \overline{\nu}_{\J} $.
\end{Definition}

This generalizes the classical definition of projective equivalence of ideals. Ideals $I$ and $J$ are projectively equivalent if there exists $\alpha\in \RR_{>0}$ such that $\overline{\nu}_I=\alpha \hspace{0.02in} \overline{\nu}_J$. Proposition \ref{prop23} in the introduction gives the beautiful classical theorem characterizing projectively equivalent ideals. Proposition \ref{prop23} is generalized to filtrations in Theorem \ref{thm25}.

Suppose that $I$ and $J$ are ideals in a Noetherian ring $R$ and $\I = \{I^m\}_{m \in \NN}$ and $\J = \{J^m\}_{m \in \NN}$ are their associated adic-filtrations. We have that $\overline{\nu}_{I} = \overline{\nu}_{\I}$ and $\overline{\nu}_J = \overline{\nu}_{\J}$, so the ideals $I$ and $J$ are projectively equivalent if and only if the associated adic-filtrations $\I$ and $\J$ are projectively equivalent.

Theorem \ref{thm23} shows that given any $\a \in \RR_{>0}$, there are projectively equivalent filtrations $\I$ and $\J$ in a ring with $\overline{\nu}_{\I} = \a \hspace{0.02in} \overline{\nu}_{\J}$. Thus, the conclusion of the rationality of $\a$ (for projectively equivalent ideals commented after Proposition \ref{prop23}) does not extend to filtrations.

We provide the following necessary and sufficient condition for projective equivalence of filtrations.
\begin{Theorem}\label{thm25}
Let $\I = \{I_m\}_{m \in \NN}$ and $\J = \{J_m\}_{m \in \NN}$ be filtrations in a Noetherian ring $R$. Then $\I$ and $\J$ are projectively equivalent if and only if $\exists$ $\alpha$, $\beta$ $\in \RR_{>0}$ such that $\mathcal{IC}(\I^{(\alpha)}) = \mathcal{IC}(\J^{(\beta)})$, or equivalently, $\overline{R[\I^{(\a)}]} = \overline{R[\J^{(\beta)}]}$.
\end{Theorem}
\begin{proof}
Suppose $ \exists $ $\a, \beta \in \RR_{>0}$ such that $\mathcal{IC}(\I^{(\a)}) = \mathcal{IC}(\J^{(\beta)})$. By Theorems \ref{thm23} and \ref{finase}, $\overline{\nu}_{\I} 
= \a \hspace{0.02in} \overline{\nu}_{\I^{(\alpha)}} 
= \a \hspace{0.02in} \overline{\nu}_{\mathcal{IC}(\I^{(\alpha)})}
= \a \hspace{0.02in} \overline{\nu}_{\mathcal{IC}(\J^{(\beta)})}
= \a \hspace{0.02in} \overline{\nu}_{\J^{(\beta)}} 
= \a \hspace{0.02in} \dfrac{\overline{\nu}_{\J}}{\beta}$. 
This shows that $\I$ and $\J$ are projectively equivalent.

Assume $\I$ and $\J$ are projectively equivalent, that is, $\exists$ $\g \in \RR_{>0}$ such that $\overline{\nu}_{\I} = \g \hspace{0.02in} \overline{\nu}_{\J} $. Choose $\alpha$, $\beta$ $\in \RR_{>0} \setminus \QQ$ such that $\dfrac{\alpha}{\beta} = \g$, or, $\alpha = \beta \g$. We show that $\mathcal{IC}(\I^{(\beta \g)}) = \mathcal{IC}(\J^{(\beta)})$. Using Lemma \ref{lem19}, $\mathcal{IC}(\I^{(\beta \g)}) = \{K_m\}_{m \in \NN}$ where $$ K_m = \{f \in R \hspace{0.05in} | \hspace{0.05in} f^r \in \overline{I^{(\beta \g)}_{rm}} = \overline{I_{\lceil \beta \g rm \rceil}} \text{ for some } r > 0\} $$
 and $\mathcal{IC}(\J^{(\beta)}) = \{L_m\}_{m \in \NN}$ where 
 $$
 L_m = \{f \in R \hspace{0.05in} | \hspace{0.05in} f^t \in \overline{J^{(\beta)}_{tm}} = \overline{J_{\lceil \beta t m \rceil}} \text{ for some } t > 0\}.
 $$
 Recall the filtrations $\overline{\mathcal I} = \{\overline{I_m}\}_{m \in \NN}$ and $\overline{\mathcal J} = \{\overline{J_m}\}_{m \in \NN}$ defined in Corollary \ref{ICcor}.

Let $x \in K_m$, that is, $x^r \in \overline{I_{\lceil\beta\gamma rm\rceil}}$ for some $r > 0$. Then $ \forall $ $ i \in \NN$, $x^{ri} \in (\overline{I_{\lceil\beta\gamma rm\rceil}})^i \subseteq \overline{I_{\lceil \beta \g r m \rceil i}}$, which implies $\nu_{\overline{\I}}(x^{ri}) \geq \lceil \beta \g rm\rceil i$. This gives $\lim\limits_{i \to \infty} \dfrac{\nu_{\overline{\I}}(x^{ri})}{ri} \geq \lim\limits_{i \to \infty} \dfrac{\lceil \beta \g rm\rceil i}{ri}$, that is, $\overline{\nu}_{\overline{\I}}(x) \geq \dfrac{\lceil \beta \g rm\rceil}{r}$. By the assumption, 
$\overline{\nu}_{\overline{\J}}(x) \geq \dfrac{\lceil \beta \g rm\rceil}{r\g}$, that is, $\lim\limits_{i \to \infty} \dfrac{\nu_{\overline{\J}}(x^i)}{i} \geq \dfrac{\lceil \beta \g rm\rceil}{r\g}$.

Suppose $\lim\limits_{i \to \infty} \dfrac{\nu_{\overline{\J}}(x^i)}{i} = \dfrac{\lceil \beta \g rm\rceil}{r\g}$. Then, given $\epsilon > 0$, $\exists$ $n_0 = n_0(\epsilon) \in \ZZ_{>0}$ such that $- \epsilon <  \dfrac{\nu_{\overline{\J}}(x^i)}{i} - \dfrac{\lceil \beta \g rm\rceil}{r\g} < \epsilon$ $\forall$ $ i \geq n_0$. For $\epsilon = \dfrac{\lceil \beta \g rm\rceil}{r\g} - \beta m > 0$, let $i_0 = n_0 (\epsilon)$.
We have that $\epsilon$ is positive since $\beta\not\in \QQ$. 
So, $-\epsilon = \beta m  - \dfrac{\lceil \beta \g rm\rceil}{r \g} < \dfrac{\nu_{\overline{\J}}(x^{i_0})}{i_0} - \dfrac{\lceil \beta \g rm\rceil}{r\g}$ implying $\nu_{\overline{\J}}(x^{i_0}) > \beta m i_0$, or that, $\nu_{\overline{\J}}(x^{i_0}) \geq \lceil \beta  i_0 m\rceil$. This shows that $x^{i_0} \in \overline{J_{\lceil \beta  i_0 m\rceil}}$, that is, $x \in L_m$. 

If $\lim\limits_{i \to \infty} \dfrac{\nu_{\overline{\J}}(x^i)}{i} > \dfrac{\lceil \beta \g rm\rceil}{r\g}$, $\exists$ $j_0 \in \ZZ_{>0}$ such that $\dfrac{\nu_{\overline{\J}}(x^{j_0})}{j_0} > \dfrac{\lceil \beta \g rm\rceil}{r\g}$. This shows that $\nu_{\overline{\J}}(x^{j_0}) > \dfrac{j_0 \lceil \beta \g rm\rceil}{r\g} \geq \dfrac{j_0 \beta \g rm}{r\g}$ implying that $\nu_{\overline{\J}}(x^{j_0}) \geq \lceil j_0 \beta m \rceil$, so,  $x^{j_0} \in \overline{J_{\lceil j_0 \beta m \rceil}}$, or that, $x \in L_m$. Thus, we have shown that $K_m \subseteq L_m$ $ \forall $ $m \in \NN$. Similarly, we can show that $L_m \subseteq K_m$ $ \forall $ $m \in \NN$, proving that $\mathcal{IC}(\I^{(\beta\gamma)}) = \mathcal{IC}(\J^{(\beta)})$.
\end{proof}

\begin{Remark}

In the proof above, the condition of $\alpha$ and $\beta$ being real numbers cannot be weakened, as shown in the following example, where $\I$ and $\J$ are projectively equivalent filtrations with $\overline{\nu}_{\I} = \overline{\nu}_{\J}$ but for no $\alpha$ or $\beta$ $\in \QQ_{>0}$ do we have that $\overline{R[\I^{(\a)}]} = \overline{R[\J^{(\beta)}]}$. This follows from (\ref{eqEx3}) below, since $\overline{R[\I^{(\a)}]} = \overline{R[\J^{(\beta)}]}$ implies $\alpha=\beta$ by Theorems \ref{finase} and \ref{thm23}.
\end{Remark}

\begin{Example}\label{ex26} We give some calculations of twists of filtrations in the power series ring in one variable over a field.
\end{Example}

Let $R = k[[x]]$, a power series ring in the variable $x$ over a field $k$. For $f \in R$, let $ord_{k[[x]]}(f) = \min\{r \hspace{0.05in} | \hspace{0.05in} f_r \neq 0 \}$ where  $f = \sum\limits_{m=0}^{\infty} f_m x^m$ with $f_m \in k$. 

Fix $c \in \ZZ_{>0}$. Let $\I = \{I_m\}_{m \in \NN}$ be given by $I_0 = R$ and let $I_m = (x^{m+c})$ for $m \in \ZZ_{>0}$ and $\J = \{J_m\}_{m \in \NN}$ be given by $J_0 = R$ and $J_m = (x^{m})$ for $m \in \ZZ_{>0}$. Both $\I$ and $\J$ are filtrations in $R$ and $\I \subseteq \J$.

Let $f \in R$ with $ord_{k[[x]]}(f) = c_0$. For $n \in \ZZ_{>0}, \nu_{\J}(f^n) = n c_0$ and  $\nu_{\I}(f^n) = n c_0 - c$ (if $nc_0 > c$) and $=0$ (if $nc_0 \leq c$). 
$$\overline{\nu}_{\I}(f) = \lim\limits_{n \to \infty} \dfrac{\nu_{\I}(f^n)}{n} = \lim\limits_{n \to \infty} \dfrac{nc_0 -c}{n} = c_0 \text{ and } \overline{\nu}_{\J}(f) = \lim\limits_{n \to \infty} \dfrac{\nu_{\J}(f^n)}{n} =
\lim\limits_{n \to \infty} \dfrac{nc_0}{n} = c_0$$
Thus $\I$ and $\J$ are projectively equivalent with $\overline{\nu}_{\mathcal I}=\overline{\nu}_{\mathcal J}$.

Observe that $\forall$ $\alpha\in \RR_{>0}$, $\overline{R[\I^{(\a)}]} \subseteq \overline{R[\J^{(\a)}]}$ in $R[t]$. We will show that
\begin{equation}\label{eqEx1}
R[\J^{(\alpha)}]\mbox{ is integrally closed in $R[t]$ $\forall$ $\alpha \in \RR_{>0}$}.
\end{equation}
\begin{equation}\label{eqEx2}
R[\J^{(\alpha)}] \subseteq \overline{R[\I^{(\alpha)}]}\mbox{ when $\alpha \in
\RR_{>0} \setminus \QQ$, proving that $\overline{R[\I^{(\alpha)}]} = \overline{R[\J^{(\alpha)}]}$}.
\end{equation}
\begin{equation}\label{eqEx3}
\overline{R[\I^{(\alpha)}]} \subsetneqq R[\J^{(\alpha)}] \mbox{ when $\alpha \in \QQ$}.
\end{equation}

To prove (\ref{eqEx1}), it is enough to show it for homogeneous elements in $R[t]$. Let $ft^n \in R[t]$ be integral over $R[\J^{(\a)}]$ with $ord_{k[[x]]}(f) = c_0$, and $n \in \NN$. If $n = 0$ or $c_0 \geq \lceil \a n \rceil$, then $ft^n \in R[\J^{(\a)}]$. If $c_0 < \lceil \a n \rceil$, since $ft^n$ is integral over $R[\J^{(\a)}]$, we have the following homogeneous equation for some $d \in \ZZ_{>0}$.
$$ (ft^n)^d + a_1 t^n (ft^n)^{d-1} + \cdots + a_i t^{ni} (ft^n)^{d-i} + \cdots + a_{d-1} t^{n(d-1)} (ft^n) + a_dt^{nd} = 0 $$
where $a_i \in J_{ni}^{(\a)} = (x^{\lceil \a n i \rceil})$ $\forall$ $1 \leq i \leq d$.

In particular, the coefficient of $t^{nd} = 0$, that is, 
$$f^d + a_1 f^{d-1} + \cdots + a_i f^{d-i} + \cdots + a_{d-1} f + a_d = 0.$$
 Since $c_0 < \lceil \a n \rceil$, $c_0 < \a n$. However $ord_{k[[x]]}(f^d) = dc_0$ but 
 $$
 ord_{k[[x]]}(a_i f^{d-i}) \geq \lceil \a ni\rceil + (d-i)c_0 \geq \a ni + dc_0 - ic_0 > dc_0.
 $$
  Therefore, the above equation is not possible. Thus, if $ft^n$ is integral over  $R[\J^{(\a)}]$, then $ft^n \in R[\J^{(\a)}]$, proving that $R[\J^{(\a)}]$ is integrally closed. 
 
To prove (\ref{eqEx2}), consider a homogeneous element in $R[\J^{(\a)}]$, say, $ft^n$ where $ord_{k[[x]]}(f) = c_0 \geq \lceil \a n \rceil$, and $n \in \NN$, which is integral over $R[\I^{(\a)}]$. If $c_0 \geq \lceil \a n \rceil + c$, then $ft^n \in R[\I^{(\a)}]$. If $c_0 < \lceil \a n \rceil + c$, since $c_0 \geq \lceil \a n \rceil > \a n$, we can let $d \in \ZZ_{>0}$ be such that $d \geq \dfrac{c}{c_0 - \a n}$. Then $ft^n$ satisfies the following integral equation over $R[\I^{(\a)}]$:
$$(ft^n)^d + a_dt^{nd} = 0 $$
where $a_d = -f^d$. By our choice of $d$, $dc_0 \geq \a n d + c$ which implies $dc_0 \geq \lceil \a n d \rceil + c$. So, $ord_{k[[x]]}(f^d) \geq \lceil \a n d \rceil + c$, which shows $a_d \in I_{\lceil \a n d\rceil} = I_{nd}^{(\a)}$.

For (\ref{eqEx3}), say $\a = \dfrac{p}{q}$ for some $p, q \in \ZZ_{>0}$. Then $x^pt^q \in R[\J^{(\a)}]$ but $\notin \overline{R[\I^{(\a)}]}$. If $x^pt^q$ were integral over $R[\I^{(\a)}]$, then for some $d \in \ZZ_{>0}$ we would have that
$$ (x^p)^d + a_1 (x^p)^{d-1} + \ldots + a_i (x^p)^{d-i} + \ldots + a_{d-1} (x^p) + a_d = 0$$
where $ord_{k[[x]]}(a_i) \geq \lceil \a q i\rceil + c = pi + c$. But this is not possible since $ord_{k[[x]]}(a_i (x^p)^{d-i}) \geq pi+c+p(d-i) > pd$ $ \forall$ $1 \leq i \leq d$.

\vspace{0.05in}


\begin{Theorem}\label{thm31}
For a filtration $\I = \{I_m\}_{m \in \NN}$ of ideals in $R$, define
$$
K(\I)_m \coloneqq \{f \in R \hspace{0.05in} | \hspace{0.05in} \overline{\nu}_{\I}(f) \geq m\} \hspace{0.05in} \forall \hspace{0.05in} m \in \NN.$$
Then $\K(\I) \coloneqq \{K(\I)_m\}_{m \in \NN}$ is a filtration of ideals in $R$ and $\I \subseteq \K(\I)$. Moreover, $\overline{\nu}_{\I} = \overline{\nu}_{\K(\I)}$ and $\K(\I)$ is the unique, largest filtration $\J$ such that $\overline{\nu}_{\I}= \overline{\nu}_{\J}$.
\end{Theorem}

\begin{proof}
Denote $\K(\I)$ by $\K$ and $K(\I)_m$ by $K_m$ $ \forall$ $m \in \NN$. If $f \in I_n$ for some $n \in \NN$, then $f^i \in (I_n)^i \subseteq I_{ni}$ implying $\nu_{\I}(f^i) \geq ni$ $ \forall$ $i \in \NN$, which gives, $\overline{\nu}_{\I}(f) \geq n$. Thus, $I_n \subseteq K_n$ $ \forall$ $n \in \NN$.

Suppose $f, g \in K_n$ for some $n \in \NN$, that is, $\overline{\nu}_{\I}(f) \geq n$ and $\overline{\nu}_{\I}(g) \geq n$. 

For $\epsilon > 0$, $ \exists$ $m_0 \in \ZZ_{>0}$ such that $\forall$ $m \geq m_0$, $\dfrac{\nu_{\I}(f^m)}{m} \geq n - \epsilon$ and $\dfrac{\nu_{\I}(g^m)}{m} \geq n - \epsilon$. Note that $ \forall $ $ m, k \in \ZZ_{>0}$, $\nu_{\I}((f+g)^{mk}) \geq \min\limits_{i+j = mk} \{\nu_{\I}(f^i g^j)\}$, using Remark \ref{rmk3.3}. Since $i + j = mk, mk \geq m \bigg\lfloor  \dfrac{i}{m} \bigg\rfloor + m \bigg\lfloor  \dfrac{j}{m} \bigg\rfloor \geq mk - 2m$. Using Remark \ref{rmk3.3} again,
$$\nu_{\I}(f^ig^j) \geq \nu_{\I}(f^i) + \nu_{\I}(g^j) \geq \nu_{\I}(f^{m \lfloor  \frac{i}{m} \rfloor }) + \nu_{\I}(g^{m\lfloor  \frac{j}{m} \rfloor }) \geq \bigg\lfloor  \dfrac{i}{m} \bigg\rfloor \nu_{\I}(f^m) + \bigg\lfloor  \dfrac{j}{m} \bigg\rfloor \nu_{\I}(g^m).$$

For $m \geq m_0$, $\nu_{\I}(f^ig^j)  \geq \bigg\lfloor  \dfrac{i}{m} \bigg\rfloor m(n - \epsilon) + \bigg\lfloor  \dfrac{j}{m} \bigg\rfloor m(n - \epsilon) \geq (mk - 2m)(n - \epsilon)$. 

Thus, for $k >> 0$,
$$ \dfrac{\nu_{\I}((f+g)^{mk})}{mk} \geq \dfrac{mk(n - \epsilon) - 2m (n - \epsilon)}{mk} = n - \epsilon - \dfrac{2}{k} (n - \epsilon).$$
Taking limits as $k \to \infty$, we get $\overline{\nu}_{\I}(f + g) \geq n - \epsilon$. 
Since $\epsilon$ is arbitrary, $\overline{\nu}_{\I}(f + g) \geq n$. This proves that $f + g \in K_n$ $ \forall$ $f,g \in K_n$.

For $r \in R$ and $f \in K_n$, using Remark \ref{rmk3.3}
$$\overline{\nu}_{\I}(rf) = \lim\limits_{i \to \infty} \dfrac{\nu_{\I}((rf)^i)}{i} \geq \lim\limits_{i \to \infty} \dfrac{\nu_{\I}(r^i)}{i} + \lim\limits_{i \to \infty} \dfrac{\nu_{\I}(f^i)}{i} = \overline{\nu}_{\I}(r) +
\overline{\nu}_{\I}(f) \geq n.$$ This shows that $rf \in K_n$, thus proving that $K_n$ is an ideal in $R$.

Clearly $K_0 = R$. If $f \in K_{n+1}$ for some $n \in \NN$, that is, $\overline{\nu}_{\I}(f) \geq n+1 > n$, then $f \in K_n$ proving that $K_{n+1} \subseteq K_n$ $ \forall$ $n \in \NN$.

Suppose $f \in K_m$ and $g \in K_n$ for some $m, n \in \NN$. Then by Remark \ref{rmk3.3}
$$\overline{\nu}_{\I}(fg) = \lim\limits_{i \to \infty} \dfrac{\nu_{\I}((fg)^i)}{i} \geq \lim\limits_{i \to \infty} \dfrac{\nu_{\I}(f^i)}{i} + \lim\limits_{i \to \infty} \dfrac{\nu_{\I}(g^i)}{i} \geq n + m.$$
Thus, $K_n K_m \subseteq K_{n+m}$. This proves $\K$ is a filtration of ideals in $R$.

For $f \in R$ and $i \in \NN$, it follows from the definition that $\overline{\nu}_{\I}(f^i) \geq \nu_{\K}(f^i)$ since $f^i \in K_{\nu_{\K}(f^i)}$. Using Proposition \ref{prop14}, $\overline{\nu}_{\I}(f) \geq \dfrac{\nu_{\K}(f^i)}{i}$ $ \forall$ $i \in \ZZ_{>0}$. Thus, $\overline{\nu}_{\I}(f) \geq \overline{\nu}_{\K}(f)$ $ \forall $ $f \in R$. Since $\I \subseteq \K$, $\overline{\nu}_{\I} \leq \overline{\nu}_{\K}$ by Remark \ref{rmk2.7}. This proves $\overline{\nu}_{\I} = \overline{\nu}_{\K}$.

For any filtration $\MC{L}$ of $R$, let $\K(\MC{L}) = \{K(\MC{L})_m\}_{m \in \NN}$ where $K(\MC{L})_m = \{f \in R \hspace{0.1in} | \hspace{0.1in} \overline{\nu}_{\L}(f) \geq m\}$. If $\J$ is a filtration such that $\overline{\nu}_{\I} = \overline{\nu}_{\J}$, then $\J \subseteq \K(\J) = \K(\I)$. This shows that every filtration $\J$ such that $\overline{\nu}_{\I} = \overline{\nu}_{\J}$ is contained in $\K(\I)$, and we have shown earlier that $\overline{\nu}_{\K(\I)} = \overline{\nu}_{\I}$, proving that $\K(\I)$ is the unique, largest filtration $\J$ such that $\overline{\nu}_{\I} = \overline{\nu}_{\J}$.
 \end{proof}
 
\begin{Example}\label{Newex} In general, the integral closure $\MC{IC}(\I)$ (so that $R[\MC{IC}(\I)]=\overline{R[\I]}$)  of a filtration $\I$ is strictly smaller than $\K(\I)$ (or equivalently $\overline{R[\I]}$ is strictly smaller than $R[\K(\I)]$). For adic-filtrations, however, $\MC{IC}(\I) = \K(\I)$, by Lemma \ref{LemId}, where $\I = \{I^m\}_{m \in \NN}$ for an ideal $I$ of $R$.
\end{Example}

\begin{proof} We consider filtrations $\I = \{I_m\}_{m \in \NN}$ where $I_0 = R$ and $I_m = (x^{m+c})$ for some fixed $c \in \ZZ_{>0}$ and $\J = \{J_m\}_{m \in \NN}$ where $J_0 = R$ and $J_m = (x^m)$ for $m \in \ZZ_{>0}$ in $R = k[[x]]$, a power series ring in the variable $x$ over a field $k$. We showed in Example \ref{ex26} that $\overline{\nu}_{\I}=\overline{\nu}_{\J}$. Thus $\mathcal K(\mathcal I)=\mathcal K(\mathcal J)$. By a direct calculation, $\mathcal K(\mathcal J)=\mathcal J$.
 Thus $\K(\I) = \J$.
  However we have shown in Example \ref{ex26} that the integral closure $\overline{R[\mathcal I]}$ of the filtration $\I$ is a proper subset of $\overline{R[\J]} = R[\J]$.
\end{proof}

\begin{Theorem}\label{NewThm}
Suppose $\I$ and $\J$ are filtrations of a Noetherian ring $R$. Then $\I$ is projectively equivalent to $\J$ with $\overline{\nu}_{\I} = \a \hspace{0.02in} \overline{\nu}_{\J}$ if and only if $\K(\I^{(\a)}) = \K(\J)$.
\end{Theorem}

\begin{proof}
$\overline{\nu}_{\I} = \alpha \hspace{0.02in} \overline{\nu}_{\J}$ if and only if $\overline{\nu}_{\I^{(\a)}} =  \overline{\nu}_{\J}$ (by Theorem \ref{thm23}), which holds if and only if $\K(\I^{(\a)}) = \K(\J)$ by Theorem \ref{thm31}.
\end{proof}

\begin{Lemma}\label{KIntcl}
For a filtration $\I$ and the corresponding filtration $\K(\I)$ (as defined in Theorem \ref{thm31}) in a Noetherian ring $R$, the Rees algebra $R[\K(\I)]$ is integrally closed in $R[t]$.
\end{Lemma}

\begin{proof}
It suffices to prove the result for homogeneous elements in $R[t]$. Let $ft^n$ be a homogeneous element in $R[t]$ that is integral over $R[\K(\I)]$. Suppose $ft^n$ satisfies the following homogeneous equation of degree $d > 0$:
$$(ft^n)^d + a_1 t^n (ft^n)^{d-1} +  \cdots + a_i t^{ni} (ft^n)^{d-i} + \cdots + a_{d-1}t^{n(d-1)} (ft^n) + a_d t^{nd} = 0 $$
where $a_i \in K(\I)_{ni}$, that is, $\overline{\nu}_{\I}(a_i) \geq ni$ $ \forall$ $ 1 \leq i \leq d$. That gives
$$f^d + a_1 f^{d-1} + \cdots + a_i f^{d-i} + \cdots + a_{d-1} f + b_d = 0.$$
If $ \hspace{0.05in} \overline{\nu}_{\I}(f) < n$, then the above equation is not possible since $\overline{\nu}_{\I}(f^d) = d \hspace{0.03in} \overline{\nu}_{\I}(f)$ but $\overline{\nu}_{\I}(a_i f^{d-i}) > d \hspace{0.02in} \overline{\nu}_{\I}(f)$ $ \forall $ $ 1 \leq i \leq d$ since 
$$
\overline{\nu}_{\I}(a_i f^{d-i}) \geq \overline{\nu}_{\I}(a_i) + (d-i) \hspace{0.03in}\overline{\nu}_{\I}(f) > ni + d \hspace{0.02in} \overline{\nu}_{\I}(f) - ni = d\overline{\nu}_{\I}(f).
$$
 If $\overline{\nu}_{\I}(f) \geq n$, then $ft^n \in R[\K(\I)]$, thus, proving that $R[\K(\I)]$ is integrally closed in $R[t]$.
\end{proof}

\section{Discrete valued filtrations}\label{DVF}

Let $R$ be a Noetherian ring. Let $P$ be a minimal prime ideal of $R$ and let $v$ be a valuation of the quotient field $\kappa(P)$ of $R/P$ which is nonnegative on $R/P$. Let 
$$
\Gamma_v=\{v(f)\mid f\in \kappa(P)^{\times}\}
$$
 be the value group of $v$ and $\MC{O}_v$ be the valuation ring of $v$ with maximal ideal $m_v$. Note that $R/P \subseteq \MC{O}_{v}$. Let $\pi: R \to R/P$ be the natural surjection. We define a map $\tilde{v} : R \to \Gamma_v \cup \{\infty\}$ by
$$
\tilde{v}(r) =
    \begin{cases}
        v({\pi(r)}) & \text{if } r \notin P\\
       \hspace{0.25in} \infty & \text{if } r \in P.
    \end{cases} 
$$
We extend the order on $\Gamma_v$ to  $\Gamma_v \cup \{\infty\}$ by requiring that 
$\infty$ has order larger than all elements of $\Gamma_v$ and $\infty + \infty = g + \infty = \infty \hspace{0.1in} \forall \hspace{0.1in} g \in \Gamma_v$.

The map $\tilde{v}$ satisfies the following properties:
$$\tilde{v}(r \cdot s) = \tilde{v}(r) + \tilde{v}(s), \hspace{0.35in} \tilde{v}(r + s) \geq \min\{\tilde{v}(r),\tilde{v}(s)\}, \hspace{0.35in} \tilde{v}^{-1}(\infty) = P.$$

We will  call $\tilde v$ a \textbf{valuation} on $R$. By abuse of notation, we will denote $\tilde{v}$ by $v$. 
If $v$ is a discrete valuation of rank 1, we will say that $v$ is a \textbf{discrete valuation} of $R$. We can then naturally identify $\Gamma_{v}$ with $\ZZ$, by identifying the element of $\Gamma_{v}$ with least positive value with $1 \in \ZZ$.

We define two valuations $v$ and $\omega$ of $R$ to be \textbf{equivalent} if $v^{-1}(\infty) = \omega^{-1}(\infty)$ and the valuations $v$ and $\omega$ on $\kappa(P)$, where $P=v^{-1}(\infty) = \omega^{-1}(\infty)$, are equivalent. In particular, since we have identified the value groups with $\ZZ$, if $v$ and $\omega$ are rank 1 discrete valuations, then  $v$ and $\omega$ are equivalent if and only if they are equal, that is, $v=\omega$.

\vspace{0.05in}

Suppose that $v$ is a discrete valuation of $R$. For $m\in \NN$, define \textbf{valuation ideals}
$$
I(v)_m=\{f\in R \| v(f)\ge m\} = \pi^{-1} \left(m_{v}^m\cap R/P\right).
$$
 An \textbf{integral  discrete valued filtration} of $R$  is a filtration $\I=\{I_m\}_{m \in \NN}$ such that  there exist discrete valuations $v_1,\ldots,v_s$ of $R$ and $a_1,\ldots,a_s\in \ZZ_{>0}$ such that 
$$
I_m = I(v_1)_{ma_1} \cap \cdots \cap I(v_s)_{ma_s} \hspace{0.05in} \forall \hspace{0.05in} m \in \NN.
$$
$\I$ is called an \textbf{$\RR$-discrete valued filtration} if $a_1, \ldots, a_s \in \RR_{>0}$ and $\I$ is called a \textbf{$\QQ$-discrete valued filtration} if $a_1, \ldots, a_s \in \QQ_{>0}$. If $a_i\in \RR_{>0}$, then
$$
I(v_i)_{ma_i}=\{f\in R \hspace{0.02in} \mid \hspace{0.02in} v_i(f)\ge ma_i\}=I(v_i)_{\lceil ma_i\rceil}.
$$
 We also call an $\RR$-discrete valued filtration a \textbf{discrete valued filtration}. If the discrete valuations $v_i$ are divisorial valuations of $\kappa(P_i)$, where $P_i$ are minimal primes of $R$,  then $\I$ is called a \textbf{divisorial filtration} of $R$.

\begin{Definition} Let $\mathcal I=\{I_m\}_{m \in \NN}$ be a discrete valued filtration, which is represented as
\begin{equation}\label{eqDVI}
I_m = I(v_1)_{ma_1} \cap \cdots \cap I(v_s)_{ma_s} \hspace{0.05in} \forall \hspace{0.05in} m \in \NN.
\end{equation}
If for each $i \in \{1, \ldots, s\}$, the representation (\ref{eqDVI}) of $I_m$ is not valid for some $m$  when the term $I(\nu_i)_{a_im}$ is removed from $I_m$  then the representation of (\ref{eqDVI}) is said to be \textbf{irredundant}.
\end{Definition}

\begin{Lemma}\label{lem30}
Let $\I=\{I_m\}_{m \in \NN}$ where $I_m = I(v_1)_{ma_1} \cap \cdots \cap I(v_s)_{ma_s}$
 be a discrete valued filtration of a Noetherian ring $R$. For $f \in R \setminus \{0\}$,
$$\nu_{\I}(f) = \min \bigg\{ \bigg\lfloor \dfrac{v_{1}(f)}{a_1}\bigg\rfloor, \cdots, \bigg\lfloor \dfrac{v_{s}(f)}{a_s}\bigg\rfloor  \bigg\} \hspace{0.1in} \text{ and } \hspace{0.1in} \overline{\nu}_{\I}(f) = \min \bigg\{  \dfrac{v_{1}(f)}{a_1}, \cdots, \dfrac{v_{s}(f)}{a_s}\bigg\}.$$
\end{Lemma}

\begin{proof}
Let $\phi_{\I}(f) \coloneqq \min\limits_{1 \leq i \leq s} \bigg\{ \bigg\lfloor \dfrac{v_{i}(f)}{a_i}\bigg\rfloor \bigg\}$ and $ \overline{\phi}_{\I}(f) \coloneqq \min\limits_{1 \leq i \leq s} \bigg\{  \dfrac{v_{i}(f)}{a_i}\bigg\}$. Let $f \in R \setminus \{0\}$.

Since $f \in I_{\nu_{\I}(f)} = \bigcap\limits_{i=1}^s I(v_i)_{a_i \nu_{\I}(f)}$, $v_i(f) \geq a_i \nu_{\I}(f)$, which implies, $\nu_{\I}(f) \leq \bigg\lfloor \dfrac{v_i(f)}{a_i} \bigg\rfloor$ $ \forall $ $1 \leq i \leq s$. This shows that $\nu_{\I}(f) \leq \phi_{\I}(f)$.

For each $i \in \{1, \ldots, s\}$, $\phi_{\I}(f) \leq \bigg\lfloor \dfrac{v_{i}(f)}{a_i}\bigg\rfloor$, which implies $\phi_{\I}(f) \leq  \dfrac{v_{i}(f)}{a_i}$, that is, $v_i(f) \geq a_i \phi_{\I}(f)$.
Thus, $f \in \bigcap\limits_{i=1}^s I(v_i)_{a_i \phi_{\I}(f)} = I_{\phi_{\I}(f)}.$ This implies $\nu_{\I}(f) \geq \phi_{\I}(f)$, proving $\nu_{\I}(f) = \phi_{\I}(f)$. Now
$$\overline{\nu}_{\I}(f) = \lim\limits_{n \to \infty} \dfrac{\nu_{\I}(f^n)}{n} = \lim\limits_{n \to \infty} \dfrac{\min\limits_{1 \leq i \leq s} \bigg\{ \bigg\lfloor \dfrac{v_{i}(f^n)}{a_i}\bigg\rfloor \bigg\}}{n} = \lim\limits_{n \to \infty} \dfrac{\min\limits_{1 \leq i \leq s} \bigg\{ \bigg\lfloor \dfrac{n v_{i}(f)}{a_i}\bigg\rfloor \bigg\}}{n}$$

\vspace{0.02in}

Since $x-1 < \lfloor x \rfloor \leq x$ for any $x \in \RR$, $ \forall $ $n \in \ZZ_{>0}$ we have that

$$\dfrac{\min\limits_{1 \leq i \leq s} \bigg\{ \dfrac{n v_{i}(f)}{a_i} - 1 \bigg\}}{n} \leq \dfrac{\min\limits_{1 \leq i \leq s} \bigg\{ \bigg\lfloor \dfrac{n v_{i}(f)}{a_i}\bigg\rfloor \bigg\}}{n} \leq \dfrac{\min\limits_{1 \leq i \leq s} \bigg\{ \dfrac{n v_{i}(f)}{a_i} \bigg\}}{n}$$

Taking limits as $n \to \infty$, we get, $\overline{\nu}_{\I}(f) = \overline{\phi}_{\I}(f)$.
\end{proof}

\begin{Corollary}\label{cor36}
Let $\I$ be a discrete valued filtration of a Noetherian ring $R$. Then  $\K(\I) = \I$.
\end{Corollary}
\begin{proof}
Represent $\I = \{I_m\}_{m \in \NN}$ by $I_m = \bigcap\limits_{i=1}^s I(v_i)_{ma_i}$.
By Lemma \ref{lem30}, for any nonzero $f \in R$, $\overline{\nu}_{\I}(f) \geq m$ if and only if $ \dfrac{v_i(f)}{a_i} \geq m$, or, $v_i(f) \geq a_i m$ $ \forall \hspace{0.1in} 1 \leq i \leq s$. This is equivalent to $f \in I(v_i)_{a_i m} \hspace{0.1in} \forall \hspace{0.1in} 1 \leq i \leq s$, or that, $f \in I_m$.
\end{proof}

\begin{Lemma}\label{lem38}
If $\tilde{v}$ and $\tilde{v}'$ are discrete valuations of $R$ and $a, b \in \RR_{>0}$ are such that $\dfrac{\tilde{v}}{a} = \dfrac{\tilde{v}'}{b}$ (as functions of $R$), then $\tilde{v} = \tilde{v}'$ and $a = b$.
\end{Lemma}

\begin{proof}
Since $b \tilde{v} = a \tilde{v}'$, $\tilde{v}^{-1}(\infty) = \tilde{v}'^{-1}(\infty)$ is a common minimal prime $P$ of $R$. Thus, $\tilde{v}$ and $\tilde{v}'$ are induced by discrete valuations $v$ and $v'$ on $\kappa(P)$. Let $\pi : R \to R/P$ be the natural surjection. Suppose $\a \in \kappa(P)$ is nonzero, that is, $\a = \dfrac{\pi(f)}{\pi(g)}$ for some $f, g \in R \setminus P$. Then,
$$v(\a) = v(\pi(f)) - v(\pi(g)) = \tilde{v}(f) - \tilde{v}(g) = \dfrac{a}{b} (\tilde{v}'(f) - \tilde{v}'(g))$$
$$= \dfrac{a}{b}(v'(\pi(f)) - v'(\pi(g))) = \dfrac{a}{b} v'(\a)$$
This shows $b v = a v'$ as function of $\kappa(P)^{\times}$. Since the value groups of $v$ and $v'$ are $\ZZ$, $ \exists $ $ x, y \in \kappa(P)^{\times}$ such that $v(x) = 1$ and $v'(y) = 1$. Since $b v = a v'$,  $b = a v'(x)$ and $b v(y) = a$. This implies $a \hspace{0.01in} | \hspace{0.01in} b$ and $b \hspace{0.01in} | \hspace{0.01in} a$. Thus, $a = b$ and hence, $\tilde{v} = \tilde{v}'$.

\end{proof}

In the following theorem,  we generalize to discrete valued  filtrations the proof of uniqueness of Rees valuations for ideals given in Theorem 10.1.6 \cite{HS}.

\begin{Theorem}\label{Theorem1} Let $v_1,\ldots,v_s$ be discrete valuations of a Noetherian ring $R$. Let $a_1,\ldots,a_s\in \RR_{>0}$, and define $\omega:R\setminus \{0\}\rightarrow \RR_{\ge 0}$ by
\begin{equation}\label{eqn5}
    \omega(f)=\min \left\{\dfrac{v_1(f)}{a_1}, \cdots, \dfrac{v_s(f)}{a_s} \right\}
\end{equation}
for $f\in R$.
If no $\dfrac{v_i}{a_i}$ can be omitted from this expression, then the $v_i$ and $a_i$ are uniquely determined by the function $\omega$, up to reindexing of the $\dfrac{v_i}{a_i}$.
\end{Theorem}

\begin{proof}
We will say that the set $\{v_1,\ldots,v_s\}$ is irredundant if no $\dfrac{v_i}{a_i}$ can be removed from (\ref{eqn5}). 

If $s = 1$, the assertion follows from Lemma \ref{lem38}.
Let $s > 1$. We define $S \subseteq R$ to be $\omega$-consistent if for any $m \in \NN$ and $f_1, \ldots, f_m \in S$, $$\omega(f_1 \cdots f_m) = \sum\limits_{i=1}^m \omega(f_i).$$

For $f \in R$, let $S_f = \{f^m \| m \in \NN\}$. Then $S_f$ is $\omega$-consistent. 




Let $\mathfrak{F}$ be the set of all nonempty $\omega$-consistent subsets of $R$. Clearly $\mathfrak{F} \neq \emptyset$ since it contains the sets $S_f$ for any $f \in R$. Partially order $\mathfrak{F}$ by inclusion. Let $\{I_{\lambda}\}_{_{\lambda \in \Lambda}}$ be a chain of $\omega$-consistent subsets of $R$. Then $I = \bigcup\limits_{\lambda \in \Lambda} I_{\lambda}$ is an upper bound for this chain which is  $\omega$-consistent. Since every chain in $\mathfrak{F}$ has an upper bound in $\mathfrak{F}$, by Zorn's Lemma, $\mathfrak{F}$ has at least one maximal element. 

We provide an explicit description of all of the maximal $\omega$-consistent subsets of $R$.
For each $1 \leq i \leq s$, define the sets 
$$
\SS_i \coloneqq \bigg\{f \in R \hspace{0.1in} \big| \hspace{0.1in} \omega(f) = \dfrac{v_i(f)}{a_i}\bigg\}.
$$ 
Each $S_i$ is $\omega$-consistent and is nonempty since it contains 1.
By the irredundancy condition on the $\dfrac{v_i}{a_i}$, we have the following remark.
\begin{remark}\label{Rmk0}
For $i \neq j$, $\SS_i \nsubseteq \bigcup\limits_{j \neq i} \SS_j$.
\end{remark}






Since each $\SS_i$ is a $\omega$-consistent subset, it is contained in some maximal $\omega$-consistent subset of $R$. We show that any maximal $\omega$-consistent subset $\SS$ of $R$ is equal to one of the $\SS_i$.  Then by Remark \ref{Rmk0} it follows that $\{\SS_i \| 1 \leq i \leq s\}$ are the distinct maximal $\omega$-consistent subsets of $R$.

\vspace{0.05in}

\begin{remark}\label{Rmk1}
If $\omega(f) = \infty$ for some $f \in R$, then $v_i(f) = \infty$ $ \forall $ $1 \leq i \leq s$. Thus $f \in \SS_i$ $ \forall $ $1 \leq i \leq s$.
\end{remark} 

\vspace{0.05in}

Suppose $\SS$ is a maximal $\omega$-consistent subset of $R$, and $\SS \neq \SS_i$ for any $1 \leq i \leq s$. Then $ \exists$ $g_i \in \SS\setminus \SS_i$ $ \forall $ $1 \leq i \leq s$. Since $g_i \notin \SS_i$, $\omega(g_i) < \infty$ (by Remark \ref{Rmk1}) and $\omega(g_i) < \dfrac{v_i(g_i)}{a_i}$ $ \forall$ $1 \leq i \leq s$. Let $g = g_1 \ldots g_s$. Since $\SS$ is $\omega$-consistent, $\omega(g) = \sum\limits_{i=1}^s \omega(g_i)$, but $\sum\limits_{i=1}^s \omega(g_i) < \sum\limits_{i=1}^s \dfrac{v_j(g_i)}{a_j}=\dfrac{v_j(g)}{a_j}$
$ \forall$ $ 1 \leq j \leq s$ since $\omega(g_i) \leq \dfrac{v_k(g_i)}{a_k}$ $ \forall$ $ k \neq i$ and $\omega(g_i) < \dfrac{v_i(g_i)}{a_i}$ (when $k = i$). 
But this implies $\omega(g) < \min\limits_{1 \leq j \leq s} \bigg\{ \dfrac{v_j(g)}{a_j}\bigg\}$, contradicting the definition of $\omega$. Thus, every maximal $\omega$-consistent set $\SS$ equals to one of the sets $\{ \SS_i \| 1 \leq i \leq s\}$. 

\vspace{0.05in}

In order to prove the Theorem, we must recover the valuations $v_i$ and the numbers $a_i \in \RR_{>0}$ from the function $\omega$. Since each $\dfrac{v_i}{a_i}$ gives a distinct maximal set $\SS_i$ and $\{v_i \| 1 \leq i \leq s\}$ is irredundant,  the number of $v_i$, which is $s$,  equals the number of distinct maximal $\omega$-consistent subsets of $R$. Fix $i \in \{1, \ldots, s\}$. 

Let $c \in R$ be such that $v_i(c) < \infty$. By Remark \ref{Rmk0}, $\exists$ $x _i \in \SS_i \setminus \bigcup\limits_{j \neq i} \SS_j$. By this property, $\forall $ $ j \neq i$, $\dfrac{v_j(x_i)}{a_j} > \dfrac{v_i(x_i)}{a_i}$. For a sufficiently large $d \in \ZZ_{>0}$, and $ \forall $ $j \neq i$ we have that
$$\left(\dfrac{v_j(x_i)}{a_j} - \dfrac{v_i(x_i)}{a_i}\right) d > \dfrac{v_i(c)}{a_i}$$

This gives $\dfrac{v_j(x_i^d)}{a_j} > \dfrac{v_i(cx_i^d)}{a_i}$, which in turn implies $\dfrac{v_i(cx_i^d)}{a_i} < \dfrac{v_j(x_i^d)}{a_j} + \dfrac{v_j(c)}{a_j} = \dfrac{v_j(cx_i^d)}{a_j}$.  This shows that $\omega(cx_i^d) = \dfrac{v_i(cx_i^d)}{a_i}$. In other words, $cx_i^d \in \SS_i$.

What we have just shown is the following remark.

\begin{remark}\label{Rmk2}
 For $c \in R$ such that $v_i(c) < \infty$ and $x_i \in \SS_i \setminus \bigcup\limits_{j \neq i} \SS_j$, $ \exists$ $d \in \ZZ_{>0}$ such that $c x_i^d \in \SS_i$. Moreover, if $c x_i^d \in \SS_i$, then $cx_i^n \in \SS_i$ $ \forall$ $n \geq d$. 
\end{remark}

Let $S = \{a \in R \| a \notin \text{ any minimal prime ideal of } R\}$. $S$ is a multiplicatively closed set and for $a \in S$,  $v_j(a) < \infty$ $ \forall$ $1 \leq j \leq s$ since each $v_j$ is infinite only on some minimal prime ideal of $R$.  In particular, $v_i(a) < \infty$ $ \forall$ $a \in S$. 

\begin{remark}\label{Rmk3}
 The construction in Remark \ref{Rmk2} applies to every element in $S$.
\end{remark}

Consider the  ring $K = S^{-1}R$. We define a function $u_i : K \to \QQ \cup \{\infty\}$ as follows:

Let $\a = f/g \in K$. Since $g \in S$, by Remark \ref{Rmk2}, $\exists $ $ e \in \ZZ_{>0} $ such that $g x_i^e
\in \SS_i$. Now, if for some $d \in \ZZ_{>0}$, $f x_i^d \in \SS_i$, then by Remark \ref{Rmk2}, we can find a sufficiently large integer $n$ such that $fx_i^n, gx_i^n \in \SS_i$ and in that case we define $u_i\left(\dfrac{f}{g} \right) \coloneqq \omega(f x_i^n) - \omega(g x_i^n)$

Otherwise, if $f x_i^d \notin \SS_i$ $ \forall $ $ d \in \ZZ_{>0}$, then we define $u_i\left(\dfrac{f}{g}\right) \coloneqq \infty$. 

\begin{remark}\label{Rmk4}
If $\dfrac{f}{g}\in K$ and     $fx_i^d \in \SS_i$ for some $d \in \ZZ_{>0}$, then $u_i \left(\dfrac{f}{g}\right) = \dfrac{v_i(f) - v_i(g)}{a_i}$ since 
$\omega(fx_i^n) = \dfrac{v_i(fx_i^n)}{a_i} = \dfrac{v_i(f) + v_i(x_i^n)}{a_i}$ and $\omega(gx_i^n) = \dfrac{v_i(gx_i^n)}{a_i} = \dfrac{v_i(g) + v_i(x_i^n)}{a_i}$. Since $g \in S$, $v_i(g) < \infty$, so, $u_i \left(\dfrac{f}{g}\right) = \infty$ if and only if $v_i(f) = \infty$.
\end{remark}

\begin{remark}\label{Rmk5}
 If $fx_i^d \notin \SS_i$  $ \forall $ $ d \in \ZZ_{>0}$, then $\omega(f) < \infty = v_i(f)$. 
\end{remark}

The remark follows because if $\omega(f) = \infty$, by Remark \ref{Rmk1}, $f \in \SS_i$ which implies $fx_i^d \in \SS_i$ for every $d \in \ZZ_{>0}$ since $\SS_i$ is multiplicatively closed, but that contradicts our assumption. Hence, $\omega(f) < \infty$. If $v_i(f) < \infty$, by Remark \ref{Rmk2}, we can find $d \in \ZZ_{>0}$ such that $f x_i^d \in \SS_i$, contradicting our assumption again. Thus, $v_i(f) = \infty$.

\vspace{0.1in}

We need to show that $u_i$ is well-defined, that is, it does not depend on the choice of $n$, $x_i$, $f$, or $g$. It follows from Remark \ref{Rmk4} that $u_i$ does not depend on $n$. 

By the definition of $u_i$ and Remarks \ref{Rmk4} and \ref{Rmk5}, $u_i \left(\dfrac{f}{g}\right) = \infty$ if and only if $v_i(f) = \infty$. So $u_i \left(\dfrac{f}{g}\right)$ is independent of the choice of $x_i \in \SS_i \setminus \bigcup\limits_{j \neq i} \SS_j$ if $v_i(f) = \infty$.

Suppose $v_i(f) < \infty$. If $x_i \in \SS_i \setminus \bigcup\limits_{j \neq i} \SS_j$, then $fx_i^d \in \SS_i$ for $d \in \ZZ_{>0}$ by Remark \ref{Rmk2}, and thus $u_i\left(\dfrac{f}{g}\right) = \dfrac{v_i(f)-v_i(g)}{a_i}$, which is independent of the choice of $x_i \in \SS_i \setminus \bigcup\limits_{j \neq i} \SS_j$. Thus, $u_i$ is independent of the choice of $x_i$.

To prove that $u_i$ doesn't depend on our choice of $f$ and $g$, we will first show the following:
$$u_i\left(\dfrac{f}{g}\right) = u_i\left(\dfrac{cf}{cg}\right) \hspace{0.05in} \forall \hspace{0.05in} c \in S.$$

Let $c \in S$, then, by Remark \ref{Rmk3}, $\exists$ $e \in \ZZ_{>0}$ such that $c x_i^e \in \SS_i$.

If $\exists$ $d' \in \ZZ_{>0}$ such that $fx_i^{d'} \in \SS_i$, then, by Remark \ref{Rmk2} and \ref{Rmk3}, $fx_i^d, gx_i^d, \in \SS_i$ for some $d \in \ZZ_{>0}$, and thus, by Remark \ref{Rmk4}, $u_i\left(\dfrac{f}{g}\right) = \dfrac{v_i(f)-v_i(g)}{a_i}$. Since $fx_i^d, gx_i^d, cx_i^e \in \SS_i$ (which is $\omega$-consistent), $fx_i^d \cdot cx_i^e = fc x_i^{d+e}, gx_i^d \cdot cx_i^e = gc x_i^{d+e} \in \SS_i$. Thus, we have
$$u_i\left(\dfrac{cf}{cg}\right) = \omega(fcx_i^{d+e}) - \omega(gcx_i^{d+e}) =
\dfrac{v_i(fcx_i^{d+e})}{a_i} - \dfrac{v_i(gcx_i^{d+e})}{a_i} = \dfrac{v_i(f) - v_i(g)}{a_i}$$

This proves $u_i\left(\dfrac{cf}{cg}\right) = u_i\left(\dfrac{f}{g}\right)$, by Remark \ref{Rmk4}.

If $fx_i^d \notin \SS_i$ for any $d \in \ZZ_{>0}$, then $u_i\left(\dfrac{f}{g}\right) = \infty$. We show that $fcx_i^d \notin \SS_i$ for any $d \in \ZZ_{>0}$, proving that $u_i\left(\dfrac{cf}{cg}\right) = \infty$.

Since $cg \in S$, by Remark \ref{Rmk3}, $\exists$ $d \in \ZZ_{>0}$ such that $cgx_i^d \in \SS_i$.  If for some $e \in \ZZ_{>0}$, $fcx_i^e \in \SS_i$, then $\omega(fcx_i^e) = \dfrac{v_i(fcx_i^e)}{a_i} = \dfrac{v_i(f) + v_i(cx_i^e)}{a_i} = \infty$ (since $v_i(f) = \infty$, by Remark \ref{Rmk5}). That implies $v_j(fcx_i^e) = \infty$ $ \forall$ $1 \leq j \leq s$. Since $c \in S$, $v_j(c) < \infty$ $ \forall $ $1 \leq j \leq s$. Thus, $v_j(fcx_i^e) = v_j(fx_i^e) + v_j(c) = \infty$. This means $v_j(fx_i^e) = \infty$ $ \forall$ $1 \leq j \leq s$, or that, $\omega(fx_i^e) = \infty$ which implies $fx_i^e \in \SS_i$ (by Remark \ref{Rmk1}), contradicting our assumption. Hence, $u_i\left(\dfrac{f}{g}\right) = u_i\left(\dfrac{cf}{cg}\right) \hspace{0.05in} \forall \hspace{0.05in} c \in S.$

Suppose $\dfrac{f}{g} = \dfrac{f'}{g'}$ in $K$. Then, $\exists$ $c \in S$ such that $c (fg' - gf') = 0$. 

Since $c, g, g' \in S$ and $S$ is multiplicatively closed, $cg, cg' \in S$. Thus, we get
$$u_i\left(\dfrac{f'}{g'}\right) = u_i \left(\dfrac{cgf'}{cgg'}\right) = u_i\left(\dfrac{cfg'}{cgg'}\right) = u_i\left(\dfrac{f}{g}\right).$$

This proves that $u_i$ is well-defined. Since the sets $\SS_i$ are determined only by $\omega$ and each $u_i$ is determined by the set $\SS_i$ and the function $\omega$, we have that the functions $u_i$ are determined only by $\omega$. So, we have a well-defined function $u_i:K= S^{-1}R\rightarrow \QQ \cup \{\infty\}$ given by
$$
 u_i\left(\dfrac{f}{g}\right) =
    \begin{cases}
        \omega(fx_i^d) - \omega(gx_i^d) & \text{if } fx_i^d \in \SS_i \text{ for some } d \in \ZZ_{>0}\\
       \hspace{0.5in} \infty & \text{if } fx_i^d \notin \SS_i \hspace{0.1in} \forall \hspace{0.1in} d \in \ZZ_{>0}
    \end{cases} 
$$

\vspace{0.05in}

\begin{remark}\label{Rmk6}
The functions $u_i$ and $\dfrac{v_i}{a_i}$ agree on $R$.
\end{remark}

The proof of the remark is as follows:
For $f \in R$, if $fx_i^d \in \SS_i$ for some $d \in \ZZ_{>0}$, then 
$$u_i\left(\dfrac{f}{1}\right) = \omega(fx_i^d) - \omega(x_i^d) = \dfrac{v_i(fx_i^d)}{a_i} - \dfrac{v_i(x_i^d)}{a_i} = \dfrac{v_i(f)}{a_i}.$$
If $fx_i^d \notin \SS_i$ $ \forall $ $ d \in \ZZ_{>0}$, then $u_i\left(\dfrac{f}{1}\right) = \infty$ which gives $v_i(f) = \infty$ (by Remark \ref{Rmk5}) and thus $\dfrac{v_i(f)}{a_i} = \infty$.

It follows from Remark \ref{Rmk6} that $u_i$ satisfies the following properties for every $f, g \in R$.
$$ u_i(fg) = u_i(f) + u_i(g) \hspace{0.1in} \text{ and } \hspace{0.1in} u_i(f+g) \geq \min\{u_i(f), u_i(g)\} $$
Let $P_i$ be the prime ideal $\{x \in R \| u_i(x) = \infty\}$ of $R$. By Remark \ref{Rmk6}, $\{f \in R \| v_i(x) = \infty\} = P_i$ and $u_i$ induces a function on $R/P_i$ which is equal to $\dfrac{v_i}{a_i}$ on $R/P_i$. Thus, $P_i$ is a minimal prime of $R$ and $u_i$ induces a valuation on the quotient field of $R/P_i$ which is equivalent to $v_i$. By abuse of notation, we will denote this valuation by $u_i$. By Remark \ref{Rmk6}, $u_i = \dfrac{v_i}{a_i}$. By Lemma \ref{lem38}, $v_i$ and $a_i$ are uniquely determined by $u_i$. Since the $u_i$ are uniquely determined by the function $\omega$, we have that the $v_i$ and $a_i$ are uniquely determined by the function $\omega$.
\end{proof}

\begin{Corollary} Let $\I=\{I_m\}_{m \in \NN}$ be a discrete valued filtration of a Noetherian ring $R$, where $I_m=I(v_1)_{a_1m}\cap \cdots \cap I(v_s)_{a_sm}$ $ \forall $ $m \in \NN$ is an irredundant representation.  Then the valuations $v_i$ and $a_i\in \RR_{>0}$ are uniquely determined.
\end{Corollary}

\begin{proof} Since $I_m=\left\{f\in R\mid v_i(f) \ge a_i m \mbox{ for } 1 \le i \le s\right\}$ and no $v_i$ can be removed from this expression, by Lemma \ref{lem30}, no $\dfrac{v_i}{a_i}$ can be removed from the expression $\overline{\nu}_{\I}(f) = \min\limits_{1 \leq i \leq s} \bigg\{  \dfrac{v_{i}(f)}{a_i}\bigg\}$. Therefore, from Theorem \ref{Theorem1} we have that $v_i$ and $a_i \in \RR_{>0}$ are uniquely determined.
\end{proof}

\begin{Corollary}\label{CorIrr}
Let $\I = \{I_m\}_{m \in \NN}$ and $\J = \{J_m\}_{m \in \NN}$ be discrete valued filtrations of a Noetherian ring $R$, where $I_m = \bigcap\limits_{i=1}^s I(v_i)_{a_i m}$ and $J_m = \bigcap\limits_{i=1}^r I(v'_{i})_{a'_i m}$ $ \forall $ $m \in \NN$ are irredundant representations. If $\overline{\nu}_{\I} = \overline{\nu}_{\J}$, then $r = s$ and after reindexing, $a_i = a'_i$ and $v_i = v'_i$.
\end{Corollary}
\begin{proof}
From Lemma \ref{lem30} we have that $\min\limits_{1 \leq i \leq s} \bigg\{  \dfrac{v_{i}(f)}{a_i}\bigg\} =  \min\limits_{1 \leq i \leq r} \bigg\{  \dfrac{v'_{i}(f)}{a'_i}\bigg\}$ $ \forall$ $f \in R$. The Corollary now follows from Theorem \ref{Theorem1}.
\end{proof}

In Definition \ref{uptwist}, we defined the filtration $\mathcal I^{(\alpha)}$, the twist of an arbitrary filtration $\mathcal I$ by $\alpha\in \RR_{>0}$.
Suppose that $\I = \{I_m\}_{m \in \NN}$ is a discrete valued filtration where $I_m=\bigcap\limits_{i=1}^s I(v_i)_{a_im}$ $ \forall $ $ m \in \NN$ and $\a \in \RR_{>0}$. Then we have the explicit description of $\mathcal I^{(\alpha)}$ as
\begin{equation*}
\I^{(\a)}=\{I_m^{(\a)}\}_{m \in \NN} = \{I_{\lceil\a m \rceil}\}_{m \in \NN} \mbox{ where }
I_{\lceil\a m \rceil} = \bigcap\limits_{i=1}^s I(v_i)_{\lceil\a m \rceil a_i} \hspace{0.05in} \forall \hspace{0.05in} m \in \NN.
\end{equation*}
We now define a new filtration
\begin{equation*}
\I^{[\a]}=\{I_m^{[\a]}\}_{m \in \NN} = \{I_{\a m}\}_{m \in \NN} \mbox{ where }
I_{\a m} = \bigcap\limits_{i=1}^s I(v_i)_{\a m  a_i}  \hspace{0.05in} \forall \hspace{0.05in} m \in \NN.
\end{equation*}

Observe that $\I^{(\a)}$ is, in general, not a discrete valued filtration, but $\I^{[\a]}$ is.

The filtration $\mathcal I^{[\alpha]}$ is well defined; that is, it is independent of (possibly redundant) representation $I_m=\bigcap\limits_{i=1}^sI(v_i)_{a_im}$ $ \forall $ $m \in \NN$. To prove this, we first show that
\begin{equation}\label{eqwd1}
I_m=\bigcap\limits_{i=1}^sI(v_i)_{a_im}  \hspace{0.05in} \forall \hspace{0.05in} m \in \NN
\end{equation}
is an irredundant representation of $\mathcal I$  if and only if
\begin{equation}\label{eqwd2}
I_m^{[\alpha]}=\bigcap\limits_{i=1}^sI(v_i)_{\alpha a_im}  \hspace{0.05in} \forall \hspace{0.05in} m \in \NN
\end{equation}
is an irredundant representation of $\mathcal I^{[\alpha]}$. This follows since (\ref{eqwd1}) is irredundant if and only if no $\dfrac{v_i}{a_i}$ can be eliminated from the function
$$
\omega(f)=\min\left\{\frac{v_1(f)}{a_1},\ldots,\frac{v_s(f)}{a_s}\right\}
$$
which holds if and only if no 
$\dfrac{\nu_i}{\alpha a_i}$ can be eliminated from the function
$$
\omega_{\alpha}(f)=\min\left\{\frac{v_1(f)}{\alpha a_1},\ldots,\frac{v_s(f)}{\alpha a_s}\right\}
$$
which is equivalent to (\ref{eqwd2}) being irredundant. Now by Corollary \ref{CorIrr}, the valuations $\nu_i$ and $a_i\in \RR_{>0}$ giving irredundant representations of $\mathcal I$ are uniquely determined and the valuations $\nu_i$ and $a_i\alpha\in \RR_{>0}$ giving irredundant representations of $\mathcal I^{[\alpha]}$ are uniquely determined. Thus the filtration $\mathcal I^{[\alpha]}$ is independent of choice of representation of $\mathcal I$.

\begin{Proposition}\label{cor37} 
Suppose that $\mathcal I$ is a discrete valued filtration of a Noetherian ring $R$ and $\alpha\in \RR_{>0}$. Then $\K(\I^{(\a)}) = \I^{[\a]} = \K(\I^{[\a]})$. 
\end{Proposition}

\begin{proof}
Since $\I^{[\a]}$ is a discrete valued filtration of $R$, by Corollary \ref{cor36}, $\I^{[\a]} = \K(\I^{[\a]})$. Now, $\K(\I^{(\a)}) = \{K(\I^{(\a)})_m\}_{m \in \NN}$, where $K(\I^{(\a)})_m = \{x \in R \| \overline{\nu}_{\I^{(\a)}}(x) \geq m\}$. For $x \in R$, $\overline{\nu}_{\I^{(\a)}}(x) \geq m$ if and only if $\overline{\nu}_{\I}(x) \geq \a m$ (by Theorem \ref{thm23}) if and only if $ x \in I_m^{[\a]}$ (by Corollary \ref{cor36}). Thus, $\K(\I^{(\a)}) = \I^{[\a]}$.
\end{proof}

\begin{Theorem}\label{Thmsquare} 
Let $\I=\{I_m\}_{m \in \NN}$ and $\J=\{J_m\}_{m \in \NN}$ be discrete valued filtrations of a Noetherian ring $R$ and $\a\in \RR_{>0}$. Then $\overline{\nu}_{\I} = \a \hspace{0.02in} \overline{\nu}_{\J}$ if and only if $\J=\I^{[\a]}$.
\end{Theorem}

\begin{proof} Theorem \ref{thm23} implies that $\overline{\nu}_{\I} = \a \hspace{0.03in} \overline{\nu}_{\I^{(\a)}}$. Thus
$\overline{\nu}_{\I} = \a \hspace{0.03in} \overline{\nu}_{\J}$ if and only if $\overline{\nu}_{\I^{(\a)}}=\overline{\nu}_{\J}$. This holds if and only if $\K(\I^{(\a)}) = \K(\J)$, by Theorem \ref{thm31}. Since $\J$ is a discrete valued filtration, by Corollary \ref{cor36}, $\K(\J) = \J$ and by Corollary \ref{cor37}, $\K(\I^{(\alpha)}) = \I^{[\alpha]}$.
\end{proof}

\section{The asymptotic Samuel function and multiplicity}

Let $R$ be a $d$-dimensional analytically irreducible (Noetherian) local ring with maximal ideal $m_R$. A filtration $\mathcal I=\{I_n\}_{n \in \NN}$ of $R$ is said to be an $m_R$-filtration if $I_n$ is $m_R$-primary $ \forall $
$n \in \ZZ_{>0}$. The multiplicity of an $m_R$-primary filtration is defined, and exists as a limit  in an analytically unramified local ring, but does not exist as a limit in general if the ring is not generically analytically unramified. This follows from Theorem 1.1 \cite{SDC2}. The multiplicity of an $m_R$-primary filtration is 
$$
e(\mathcal I)=\lim_{n\rightarrow \infty}\frac{\ell_R(R/I_n)}{n^d/d!}
$$
where $\ell_R(N)$ is the length of an $R$-module $N$. Let $K$ be the quotient field of $R$. A valuation $\nu$ of $K$ is an $m_R$-divisorial valuation of $R$ if the valuation ring $\mathcal O_{\nu}$ dominates $R$ (so that $R\subset \mathcal O_{\nu}$ and the maximal ideal $m_{\nu}$ of $\mathcal O_{\nu}$ satisfies $m_{\nu}\cap R=m_R$) and $\mathcal O_{\nu}$ is essentially of finite type over $R$. Since $\mathcal O_{\nu}$ is a Noetherian ring, a divisorial valuation is a discrete (rank 1) valuation. More about divisorial valuations can be found in Section 9.3 \cite{HS}. For $I$ an ideal in $R$, $\nu(I)=\min\{\nu(x)\mid x\in I\}$. Let $\mathcal I=\{I_n\}_{n \in \NN}$ be an $m_R$-primary filtration and let $\nu$ be an $m_R$-divisorial valuation of $R$. Then define
$$
\nu(\mathcal I)=\inf\frac{\nu(I_n)}{n}=\lim_{n\rightarrow\infty}\frac{\nu(I_n)}{n}.
$$
The existence of this limit is shown for instance in Proposition 2.3 \cite{JM}. In \cite{BLQ}, the saturation $\tilde{\mathcal I}$ of $\mathcal I$ is defined to be the $m_R$-filtration $\tilde{\mathcal I}=\{\tilde I_n\}_{n \in \NN}$ where 
$$
\tilde I_n=\{f\in m_R\mid \nu(f)\ge n\nu(\mathcal I)\mbox{ for all $m_R$-divisorial valuations of $R$}\}.
$$
It is shown in \cite{BLQ}, that in an analytically irreducible local ring $(R,m_R)$, for any filtration $\mathcal I$ of $m_R$-primary ideals, $\tilde{\mathcal I}$  is the unique largest filtration $\mathcal J$ containing $\mathcal I$ such that $\mathcal I$ and $\mathcal J$ have the same multiplicity. Specifically, they show in Theorem 4.1 \cite{BLQ} that in an analytically irreducible local ring $R$, $m_R$-filtrations $\mathcal I\subseteq \mathcal J$ satisfy $e(\mathcal I)=e(\mathcal J)$ if and only if $\tilde{\mathcal I}=\tilde{\mathcal J}$.  

\begin{Theorem}\label{Mult} Suppose that $R$ is an analytically irreducible local ring with maximal ideal $m_R$ and $\mathcal I$ and $\mathcal J$ are $m_R$-filtrations such that $\overline{\nu}_{\mathcal I}=\overline{\nu}_{\mathcal J}$. Then $e(\mathcal I)=e(\mathcal J)$.
\end{Theorem}

\begin{proof} We have that  $\mathcal I,\mathcal J\subseteq \mathcal K(\mathcal I)$ by Theorem \ref{thm31}. Thus it suffices to show that $e(\mathcal K(\mathcal I))=e(\mathcal I)$. 

Write $\mathcal K(\mathcal I)=\{K(\mathcal I)_n\}_{n \in \NN}$ and $\tilde{\I} = \{\tilde{I}_n\}_{n \in \NN}$.
Suppose that $f\in K(\mathcal I)_l$ for some $l \in \ZZ_{>0}$ and $m,n\in \ZZ_{>0}$ are such that $f^n\in I_m$. Let $\nu$ be an $m_R$-divisorial valuation of $R$. Then $\nu(f^n)\ge \nu(I_m)\ge m\nu(\mathcal I)$, so $\nu(f^n)\ge \nu_{\mathcal I}(f^n)\nu(\mathcal I)$, since for $g\in R$, $\nu_{\mathcal I}(g)=\sup\{m\mid g\in I_m\}$. Thus $\nu(f)\ge \dfrac{\nu_{\mathcal I}(f^n)}{n}\nu(\mathcal I)$ which implies
$$
\nu(f)\ge \overline{\nu}_{\mathcal I}(f)\nu(I)\ge l\nu(\mathcal I).
$$
Thus $f\in\tilde{I}_l$. We conclude that $\mathcal I\subseteq \mathcal K(\mathcal I)\subseteq\tilde{\mathcal I}$ and so $e(\mathcal K(\mathcal I))=e(\tilde{\mathcal I})$ by Theorem 4.1 \cite{BLQ}.
\end{proof}

In Theorem \ref{Mult}, we showed that if $\mathcal I$ is an $m_R$-primary filtration in an analytically irreducible local ring $R$, then $\mathcal K(\mathcal I)\subseteq \tilde{\mathcal I}$. The following example shows that the saturation $\tilde{\mathcal I}$ can be much larger that $\mathcal K(\mathcal I)$.

\begin{Example}\label{MultEx} An example showing that we can have that $\mathcal K(\mathcal I)$ is a proper subset of $\tilde{\mathcal I}$.
\end{Example}

\begin{proof} Let $k$ be a field and $R=k[[x,y]]$, a power series ring in two variables. Let $f=y-x-x^2-x^3-\cdots\in R$. Let $I_n=(x^n,f)$ for $n\in \ZZ_{>0}$. Then, setting $I_0=R$, we have that $\mathcal I=\{I_n\}_{n \in \NN}$ is an $m_R$-primary filtration.
We have that $y^m=(f+(x+x^2+\cdots))^m=fh+(x+x^2+\cdots)^m$ for some $h\in R$. Thus 
$y^m\in I_n$ if and only if $m\ge n$. We then have that 
$$
\nu_{\mathcal I}(y^m)=\max\{n\mid y^m\in I_n\}=m
$$
 for all $m$ and
 $$
 \overline{\nu}_{\mathcal I}(y^m)=\lim_{n\rightarrow\infty}\frac{\nu_{\mathcal I}(y^{nm})}{n}=m.
 $$
 Thus each $\mathcal K(\mathcal I)_m$ properly contains $\mathcal K(\mathcal I)_{m+1}$, since 
 $y^m\in \mathcal K(\mathcal I)_m\setminus \mathcal K(\mathcal I)_{m+1}$. 
 
 Let $\nu$ be an $m_R$-divisorial valuation. Then $\nu(I_n)\le \nu(f)$, so $\nu(\mathcal I)=0$. Thus $\tilde{I}_n=m_R$ for all $n>0$ where $\tilde{\I} = \{\tilde{I}_n\}_{n \in \NN}$.

\end{proof}

In the special case that $\mathcal I$ is a discrete valued filtration, we have that $\tilde{\mathcal I}=\mathcal I$.

\end{document}